\documentclass[12pt,reqno]{amsart}
\usepackage{graphicx}
\usepackage{amssymb,amsmath}
\usepackage{amsthm}
\usepackage{color}
\usepackage[pdf]{pstricks}
\usepackage{hyperref}

\numberwithin{equation}{section}
\numberwithin{figure}{section}

\newenvironment{proof1}%
{\begin{trivlist} \item[]{{\em Proof} }}%
	{\hspace*{\fill}$\rule{.3\baselineskip}{.35\baselineskip}$\end{trivlist}}

\newcommand{\beq}{\begin{equation}}
	\newcommand{\eeq}{\end{equation}}

\newtheorem{theorem}{Theorem}[section]

\newtheorem{lemma}{Lemma}[section]
\newtheorem{assumption}{Assumption}[section]
\newtheorem{corollary}{Corollary}[section]

\newtheorem{proposition}{Proposition}[section]
\newtheorem{remark}{Remark}[section]
\newtheorem*{lemma*}{Lemma}

\begin{document}
	
	\title[Morse index for the ground state]{\bf Morse index for the ground state in the energy supercritical Gross--Pitaevskii equation}
	
	\author{Dmitry E. Pelinovsky}
	\address[D.E. Pelinovsky]{Department of Mathematics and Statistics, McMaster University,
		Hamilton, Ontario, Canada, L8S 4K1}
	\email{dmpeli@math.mcmaster.ca}
	
	\author{Szymon Sobieszek}
	\address[S. Sobieszek]{Department of Mathematics and Statistics, McMaster University,	Hamilton, Ontario, Canada, L8S 4K1}
	\email{sobieszs@mcmaster.ca}

	
	
	\begin{abstract}
		The ground state of the energy super-critical Gross--Pitaevskii equation with a harmonic potential converges in the energy space to the singular solution in the limit of large amplitudes. The ground state can be represented by a solution curve which has either oscillatory or monotone behavior, depending on the dimension of the system and the power of the focusing nonlinearity. We address here the monotone case for the cubic nonlinearity in the spatial dimensions $d \geq 13$. By using the shooting method for the radial Schr\"{o}dinger operators, we prove that the 	Morse index of the ground state is finite and is independent of the (large) amplitude. It is equal to the Morse index of the limiting singular solution, which can be computed from numerical approximations. 
		The numerical results suggest that the Morse index of the ground state is one and that it is stable in the time evolution of the cubic Gross--Pitaevskii equation in dimensions $d \geq 13$.
	\end{abstract}
	
	\date{\today}
	\maketitle
	
	\section{Introduction}
	
	We consider the stationary Gross-Pitaevskii equation with a harmonic potential, 
	\begin{equation}
		\label{statGP}
		(-\Delta u +|x|^2) u - |u|^{2p} u = \lambda u\,,
	\end{equation}
	where $x \in \mathbb{R}^d$, $\lambda \in \mathbb{R}$, and $u \in \mathbb{R}$. Existence of its ground state (a positive and radially decreasing solution) has been addressed before in the energy subcritical \cite{Fuk,KW}, 	critical \cite{Selem2011}, and supercritical \cite{SK2012,Selem2013} regimes, where the critical exponent is $p = \frac{2}{d-2}$ if $d \geq 3$. We are concerned here with the energy supercritical case of the focusing Gross--Pitaevskii equation. Scattering in the defocusing version of energy supercritical equations 
	was studied in \cite{kilip1,kilip2,kilip3}.
	
The stationary Gross--Pitaevskii equation (\ref{statGP}) is the Euler--Lagrange equation for the action functional $\Lambda_{\lambda}(u) = E(u) - \lambda M(u)$, where $E(u)$ and $M(u)$ are the energy and mass given by  
	\begin{equation}
	\label{mass} 
	M(w) = \int_{\mathbb{R}^d} |u|^2 dx
	\end{equation}
	and
	\begin{equation}
	\label{energy-E} 
	E(w) = \int_{\mathbb{R}^d} \left( |\nabla u|^2 + |x|^2 |u|^2 - \frac{1}{p+1} |u|^{2p+2} \right) dx.
	\end{equation}
The energy and mass are formally the conserved quantities in the evolution of the time-dependent Gross--Pitaevskii equation. They are defined in the energy space  $\mathcal{E}\cap L^{2p+2}_r$, where 
	\begin{equation}
		\label{energy-space}
		\mathcal{E} := \left\{ u \in L^2_r(\mathbb{R}^+) : \quad u' \in L^2_r(\mathbb{R}^+), \quad r u \in L^2_r(\mathbb{R}^+) \right\}
	\end{equation}
	and $L^q_r$ denotes the space of radially symmetric $L^q(\mathbb{R}^d)$ functions. 
	
The ground state of the stationary equation (\ref{statGP}) can be obtained variationally in the energy subcritical case $p (d-2) < 2$, but the variational methods are not applicable in the energy critical and supercritical cases $p(d-2) \geq 2$ if $d \geq 3$. In all cases, 
the ground state forms a family which appears as the curve on the $(\lambda,b)$ plane, where $b := u(0) \equiv \| u \|_{L^{\infty}}$ is referred to as {\em the amplitude}. The large-amplitude limit is the limit as $b \to \infty$. In the energy supercritical regime, it was proven in \cite{Selem2013} that there exists $\lambda_{\infty} \in (0,d)$ such that $\lambda \to \lambda_{\infty}$ as $b \to \infty$ along the solution curve.  It was discovered in our recent work \cite{BIZON2021112358} that there exists $d_*(p)$ such that the solution curve is oscillatory for $2 + \frac{2}{p} < d < d_*(p)$ and monotone for $d > d_*(p)$. 
	
	For the sake of simplicity, we have been working with the cubic nonlinearity, $p = 1$, for which 
	$d_*(p=1) = 8 + 2 \sqrt{6}$. We will continue working with the cubic nonlinearity here. Figure \ref{fig:blambda} shows the dependence 
	of $\lambda(b)$, which is oscillatory for $5 \leq d \leq 12$ 
	and monotone for $d \geq 13$.
	
		\begin{figure}[htp!]
		\centering
		\includegraphics[width=0.45\textwidth]{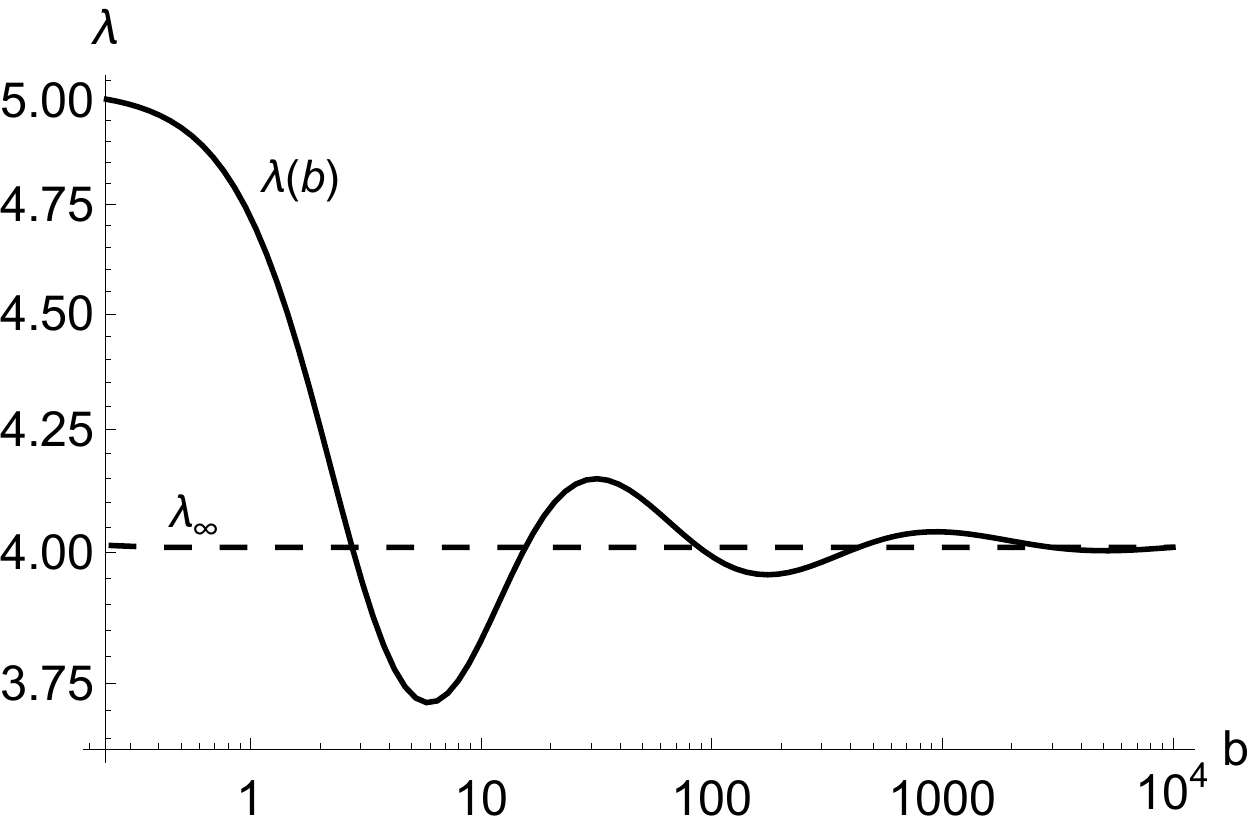}\qquad
		\includegraphics[width=0.45\textwidth]{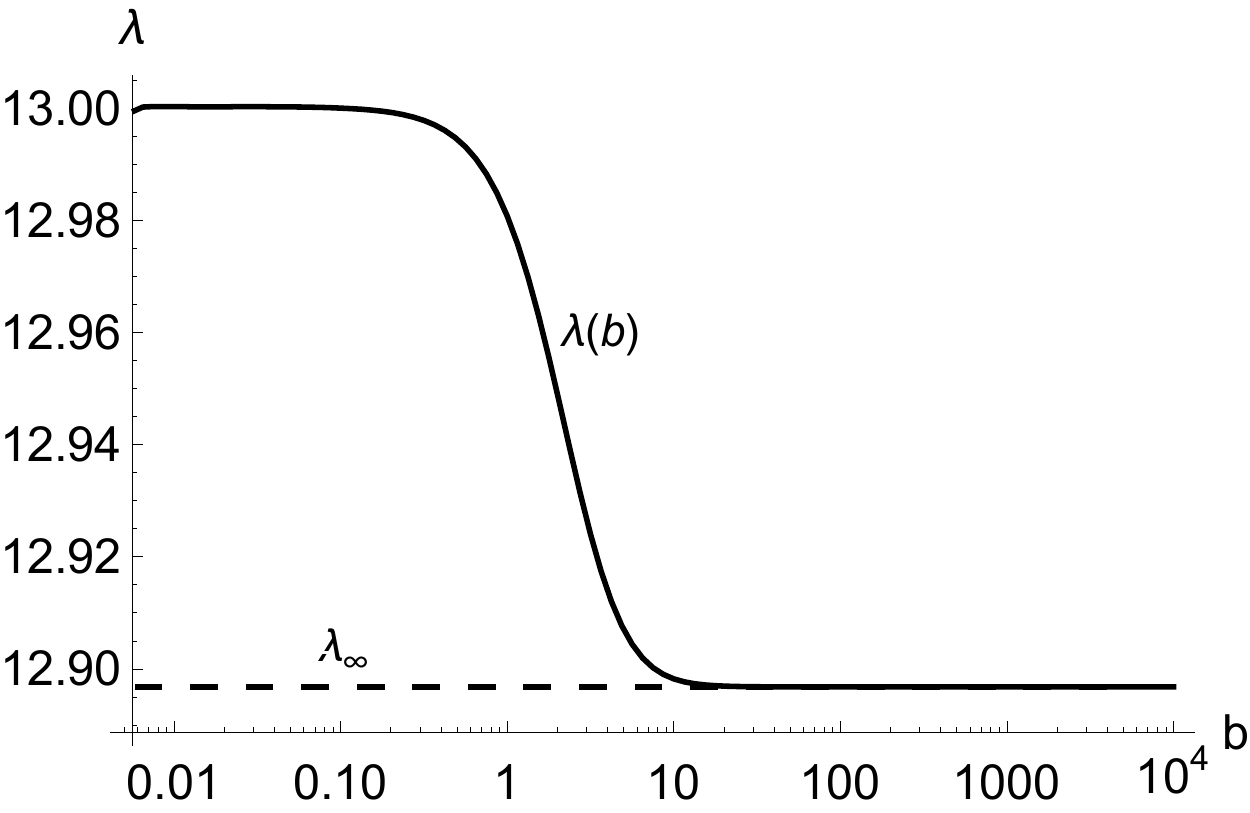}
		\caption{Graph of $\lambda$ versus $b$ for the ground state of the boundary-value problem (\ref{statGPrad}) for $d=5$ (left) and $d=13$ (right).}
		\label{fig:blambda}
	\end{figure}
	
	A similar duality between the oscillatory and monotone behaviors was 
	discovered for the classical Liouville--Bratu--Gerlfand problem in \cite{JL} 
	and explored in \cite{Budd1989,Budd_Norbury1987,DF,MP} for the stationary focusing nonlinear Schr\"odinger equation in a ball and without a harmonic potential. The similarity is explained by the same linearization of the stationary equation near the origin after the Emden--Fowler transformation \cite{F}. Another example of the oscillatory and monotone behaviors 
	was considered in \cite{Ficek} for the Schr\"{o}dinger--Newton--Hooke model.
	
	Let us define the ground state of the stationary equation (\ref{statGP}) 
	in radial variable $r = |x|$ as a solution of the following boundary-value problem:
	\begin{equation}
		\label{statGPrad}
		\left\{ \begin{array}{ll}
			\mathfrak{u}''(r) + \frac{d-1}{r} \mathfrak{u}'(r) - r^2 \mathfrak{u}(r) + \lambda \mathfrak{u}(r) + \mathfrak{u}(r)^3 = 0, \quad & r > 0, \\
			\mathfrak{u}(r) > 0, \qquad \qquad \mathfrak{u}'(r) < 0, \quad & \\
			\lim\limits_{r \to 0} \mathfrak{u}(r) < \infty, \quad 
			\lim\limits_{r \to \infty} \mathfrak{u}(r) = 0. & \end{array} \right.
	\end{equation}
	Any solution of the boundary-value problem (\ref{statGPrad}) 
	belongs to $\mathcal{E} \cap L^{\infty}$, since $\mathfrak{u} \in C^2(0,\infty)$ and $\mathfrak{u}(r) \to 0$ decays fast as $r \to \infty$.
	
	As is well understood since the pioneering work in \cite{JL}, 
	the ground state of the boundary-value problem (\ref{statGPrad}) 
	can be found from the family of solutions 
	to the following initial-value problem:
	\begin{equation}
		\label{statGPshoot}
		\left\{ \begin{array}{ll}
			f_b''(r) + \frac{d-1}{r} f_b'(r) - r^2 f_b(r) + \lambda f_b(r) + f_b(r)^3 =0,\quad & r > 0, \\
			f_b(0)=b, \quad f_b'(0)=0, & \end{array} \right.
	\end{equation}
	where $b>0$ is an arbitrary parameter. By Theorem 1.1 in our previous work \cite{BIZON2021112358}, for any $b>0$ and $d\geq 4$, there exists some $\lambda \in(d-4,d)$ such that the unique solution $f_b \in C^2(0,\infty)$ to the initial-value problem \eqref{statGPshoot} is monotonically decaying to zero as $r\to\infty$, making it a ground state $\mathfrak{u} \equiv \mathfrak{u}_b \in \mathcal{E}\cap L^{\infty}$ of the boundary-value problem \eqref{statGPrad}. That value of $\lambda$ is denoted as $\lambda(b)$. The mapping $b \mapsto \lambda(b)$ defines a solution curve on the $(\lambda,b)$ plane. Uniqueness of $\lambda(b)$ is an open problem for $d \geq 4$, whereas Figure \ref{fig:blambda} suggests that $\lambda(b)$ is unique for every $b > 0$.
	
	{\em The purpose of this work is to study the Morse index of the ground state $\mathfrak{u}_b$.} It is defined as the number of negative eigenvalues of the linearized operator $\mathcal{L}_b$ given by 
	\begin{equation}
		\label{eq:Lb_op_r}
		\mathcal{L}_b := -\frac{d^2}{dr^2} - \frac{d-1}{r} \frac{d}{dr} + r^2 - \lambda(b) -3 \mathfrak{u}_b^2(r).
	\end{equation}
	Since $\mathcal{E}$ is the form domain of $\mathcal{L}_b$, we can write 
$\mathcal{L}_b : \mathcal{E}\mapsto \mathcal{E}^*$, where $\mathcal{E}^*$ is the dual of $\mathcal{E}$ with respect to the scalar product in $L^2_r$.

	Assuming $C^1$ property of $\mathfrak{u}_b$ in $b$ and differentiating the initial-value problem (\ref{statGPshoot}) with $\lambda = \lambda(b)$ in $b$, we can see that $\mathcal{L}_b \partial_b \mathfrak{u}_b = \lambda'(b) \mathfrak{u}_b$, where $\partial_b \mathfrak{u}_b \in \mathcal{E}$. Hence, any value of $b$ for which $\lambda'(b) = 0$ corresponds to zero eigenvalue being in the spectrum 
	of $\mathcal{L}_b$ in $L^2_r$. 
	
	Although the converse is not known, this property implies 
	that the oscillatory case is very different from the monotone case, 
	where the former has infinitely many crossing of zero eigenvalue of $\mathcal{L}_b$ in the parameter continuation in $b$ as $b \to \infty$ whereas the latter does not have any eigenvalue crossing as $b \to \infty$, see also Figure \ref{fig:blambda}. This suggests 
	that the Morse index should be well defined in the monotone case, 
independently of $b$ for large values of $b$. This is in fact the main result 
	which we formulate as the following theorem. 
	
	\begin{theorem}
		\label{theorem-main} 
		For every $d \geq 13$, there exists $b_0 > 0$ such that the Morse index of $\mathcal{L}_b : \mathcal{E}\mapsto \mathcal{E}^*$ is finite and is independent of $b$ for every $b \in (b_0,\infty)$.
	\end{theorem}
	
	\begin{remark}
		Regarding the Morse index for the ground state in the energy supercrticial case, 
		we are only aware of the works \cite{GuoWei,KikuchiWei}, where the Morse index was estimated in the monotone case for the limiting singular solutions of the Dirichlet problem in a ball. We believe that the conclusion of Theorem \ref{theorem-main} and the technique behind its proof remain valid for other problems in the monotone case, e.g. for the nonlinear Schr\"{o}dinger equation in a ball. 
	\end{remark}
	
	\begin{remark}
		By the Lyapunov--Schmidt reduction technique (see, e.g., \cite{SK2012}), 
		the solution curve satisfies $\lambda(b) \to d$ and $\mathfrak{u}_b \to 0$ as $b \to 0$, where the Morse index of $\mathcal{L}_b$ in $L^2_r$ is equal to one. If the Morse index is equal to one for $b > b_0$ in Theorem \ref{theorem-main}, then it is quite possible 
		that the Morse index remains one for every $b \in (0,\infty)$. Since the ground state under these conditions is orbitally stable in the time evolution of the Gross--Pitaevskii equation in $\mathcal{E}\cap L^4_r$ if the mapping of $\lambda \mapsto \| \mathfrak{u}_b \|^2_{L^2_r}$ is monotonically decreasing (see, e.g., Theorem 4.8 in \cite{Pel-book}), it is rather interesting 
		to observe that the transition from the oscillatory case for $5 \leq d \leq 12$ to the monotone case $d \geq 13$ may enforce stability of the ground state.
	\end{remark}
	
	We shall now explain the strategy to prove Theorem \ref{theorem-main}. 
	We use the limiting singular solution $\mathfrak{u}_{\infty} \in \mathcal{E}\cap L^4_r$ which exists 
	for a particular value of $\lambda = \lambda_{\infty}$ if $d \geq 5$ \cite{Selem2013}. It was also established in \cite{Selem2013} that $\mathfrak{u}_b\to \mathfrak{u}_\infty$ in $\mathcal{E}$ and $\lambda(b)\to\lambda_\infty$ as $b\to\infty$. 
	Uniqueness of $\lambda_{\infty}$ is also an open problem, see the discussion in \cite{BIZON2021112358}. 
	
	Figure \ref{fig:ub_u_inf_sol} shows the ground state $\mathfrak{u}_b(r)$  for two values of $b$ and the limiting singular solution $\mathfrak{u}_{\infty}(r)$ for $d = 13$. The discrepancy between the two solutions moves to smaller values of $r$ if the value of $b$ is increased. When $b = 10$, the difference between $\mathfrak{u}_b$ and $\mathfrak{u}_{\infty}$ becomes invisible on the scale used in Figure \ref{fig:ub_u_inf_sol}.
	
		\begin{figure}[htp!]
		\centering
		\includegraphics[width=0.7\textwidth]{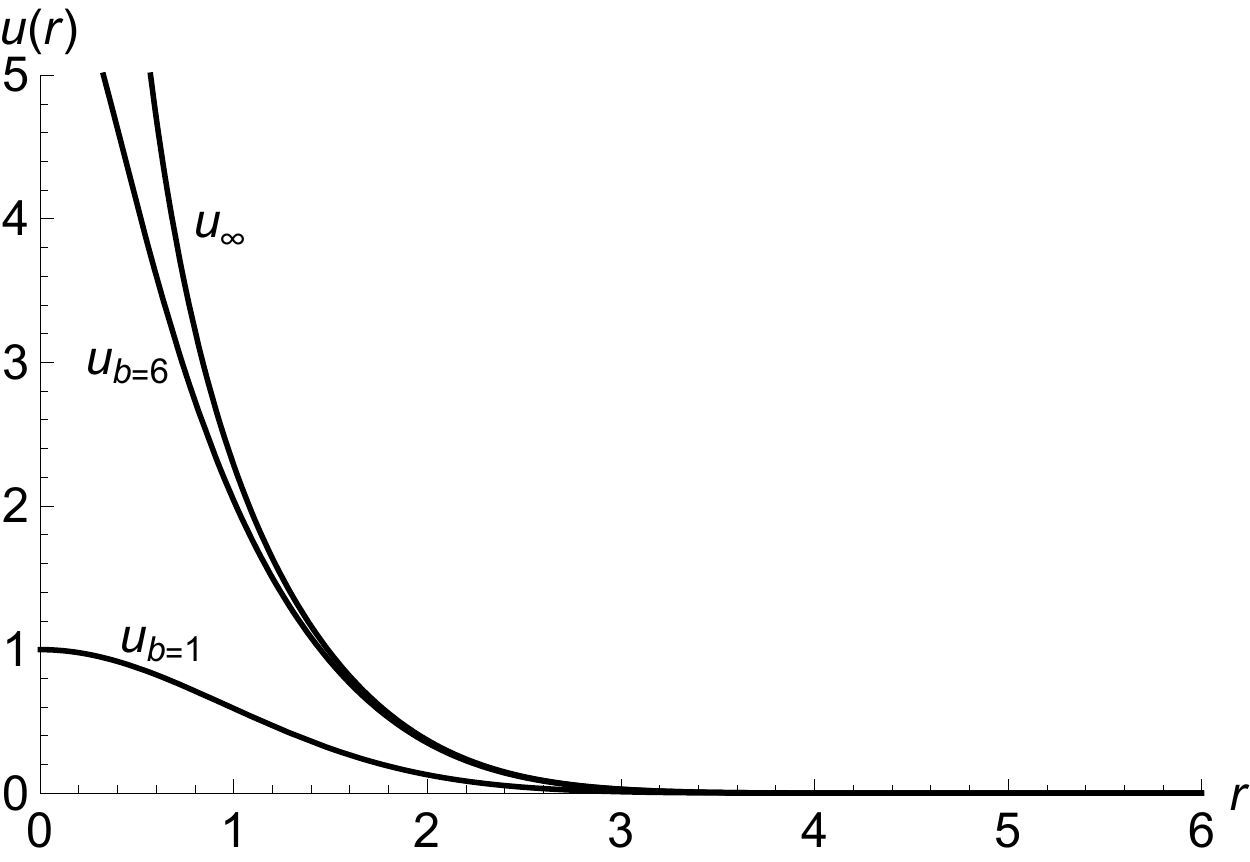}
		\caption{Graph of the ground state $\mathfrak{u}_b$ for $b = 1$ and $b = 6$ in comparison with the limiting singular solution $\mathfrak{u}_\infty$ for $d=13$.}
		\label{fig:ub_u_inf_sol}
	\end{figure}
		
	Since the limiting singular solution $\mathfrak{u}_{\infty}$ satisfies the following divergent behavior
	\begin{equation}
		\label{eq:u_inf_asymptotics}
		\mathfrak{u}_\infty(r)=\frac{\sqrt{d-3}}{r}\left[1+\mathcal{O}\left(r^2\right)\right], \quad \text{as } r\to 0,
	\end{equation}
	it is natural to introduce $F_{\infty}(r) := r \mathfrak{u}_{\infty}(r)$ and define it from the family of solutions to the following initial-value problem
	\begin{equation}
		\label{statGPsingular}
		\left\{ \begin{array}{ll}
			F''(r) + \frac{d-3}{r} F'(r) - \frac{d-3}{r^2} F(r) - r^2 F(r) + \lambda F(r) + \frac{1}{r^2} F(r)^3 = 0,\quad & r > 0, \\
			F(0) = \sqrt{d-3}, \quad F'(0)=0. & \end{array} \right.
	\end{equation}
	By Theorem 1.2 in \cite{BIZON2021112358} (based on the earlier work in \cite{Selem2013}), for any $d \geq 5$, there exists some $\lambda_{\infty} \in (d-4,d)$ such that the unique solution $F_{\infty} \in C^2(0,\infty)$ to the initial-value problem (\ref{statGPsingular}) is monotonically decreasing to zero as $r \to \infty$, making it the limiting singular solution $\mathfrak{u}_{\infty} \in \mathcal{E} \cap L^4_r$ by the transformation $\mathfrak{u}_{\infty}(r) = r^{-1} F_{\infty}(r)$.

	By Theorem 1.3 in \cite{BIZON2021112358}, convergence of $\lambda(b) \to \lambda_{\infty}$ as $b \to \infty$ is oscillatory for $5\leq d \leq 12$ and monotone for $d\geq 13$, see Figure \ref{fig:blambda}. The latter case is the only case we are interested in here. 
	
	In order to characterize the Morse index of $\mathcal{L}_b$, we use the Emden--Fowler transformation \cite{F} for the nonlinear equation in (\ref{statGPsingular}) and study two families of solutions.	
	One family is obtained from $F_b(r) := r \mathfrak{u}_b(r)$ and is parametrized 
	by its parameter $b$ from the behavior as $r \to 0$. The other family is 
	parametrized by another parameter $c$ from the decaying behavior 
	as $r \to \infty$. The second family is considered in a local 
	neighborhood of the limiting singular solution $F_{\infty}(r) = r \mathfrak{u}_{\infty}(r)$. Both families have $C^1$ property with respect to 
	their parameters and their derivatives with respect to these parameters are solutions of the homogeneous equation $\mathcal{L}_b v = 0$ 
	after the inverse Emden--Fowler transformation, e.g., $v(r) = r^{-1} \partial_b F_b(r)$. The proof of Theorem \ref{theorem-main} is achieved from the Sturm's Oscillation Theorem (see, e.g., Theorem 3.5 in \cite{simon}) 
	by showing that the two derivatives have finitely many oscillations 
	and there exists $b_0 > 0$ such that the two derivatives are linearly 
	independent for every $b \in (b_0,\infty)$. 
	
	As a by-product of our approach, we establish the equivalence of the 
	Morse index of $\mathcal{L}_b$ in $L^2_r$ with the Morse index 
	of the limiting operator $\mathcal{L}_{\infty} := \lim\limits_{b \to \infty} \mathcal{L}_b$ which is computed 
	at the limiting singular solution for $d \geq 5$:
	\begin{equation}
		\label{eq:Lb_op_infty}
		\mathcal{L}_{\infty} := -\frac{d^2}{dr^2} - \frac{d-1}{r} \frac{d}{dr} + r^2 - \lambda_{\infty} -3 \mathfrak{u}_{\infty}^2(r).
	\end{equation}
	Compared to $\mathcal{L}_b : \mathcal{E}\mapsto \mathcal{E}^*$, where the potential $-3 \mathfrak{u}_b^2(r)$ is bounded from below, 
	the potential $-3 \mathfrak{u}_{\infty}^2(r)$ is unbounded from below, 
	hence $\mathcal{L}_{\infty} : \mathcal{E}_{\infty} \mapsto \mathcal{E}_{\infty}^*$, where $\mathcal{E}_{\infty} = \{ u \in \mathcal{E} : \; r^{-1} u \in L^2_r \}$ and $\mathcal{E}_{\infty}^*$ is the dual of $\mathcal{E}_{\infty}$ with respect to the scalar product in $L^2_r$.
	
	The following theorem gives the precise result on the Morse index of the two linear operators. 
	
	\begin{theorem}
		\label{theorem-main-infty} 
		For every $d \geq 13$, there exists $b_0 > 0$ such that the Morse index of $\mathcal{L}_b : \mathcal{E}\mapsto \mathcal{E}^*$ coincides with 
		the Morse index of $\mathcal{L}_{\infty} : \mathcal{E}_{\infty} \mapsto \mathcal{E}^*_{\infty}$.
	\end{theorem}
	
	\begin{remark}
		If the norm convergence of the resolvent for $\mathcal{L}_b$ to the resolvent for $\mathcal{L}_{\infty}$ can be established as $b \to \infty$, this would imply the result of Theorem \ref{theorem-main-infty}. We do not study the norm convergence of resolvents here as our methods are based on analysis of differential equations.
	\end{remark}
	
	\begin{remark}
		The result of Theorem \ref{theorem-main-infty} suggests a simple way to obtain the Morse index for the ground state $\mathfrak{u}_b$ in the monotone case for large $b$ from the Morse index for the limiting singular solution $\mathfrak{u}_{\infty}$. The latter one can be approximated numerically with good accuracy.
	\end{remark}

Figure \ref{fig:Morse_index_Lb_Linf} shows uniquely normalized solutions 
$v(r)$ of $\mathcal{L}_b v = 0$ with $b = 1$ and $\mathcal{L}_{\infty} v = 0$ such that $v(r) \to 0$ as $r \to \infty$. Both solutions diverge as $r \to 0$ 
with different divergence rates. Since there exists only one zero for each solution on $(0,\infty)$, Sturm's Oscillation Theorem (Theorem 3.5 in \cite{simon}) asserts that the Morse index of both $\mathcal{L}_b$ and $\mathcal{L}_{\infty}$ in $L^2_r$ is equal to one.
	
	\begin{figure}[htp!]
	\includegraphics[width=0.7\textwidth]{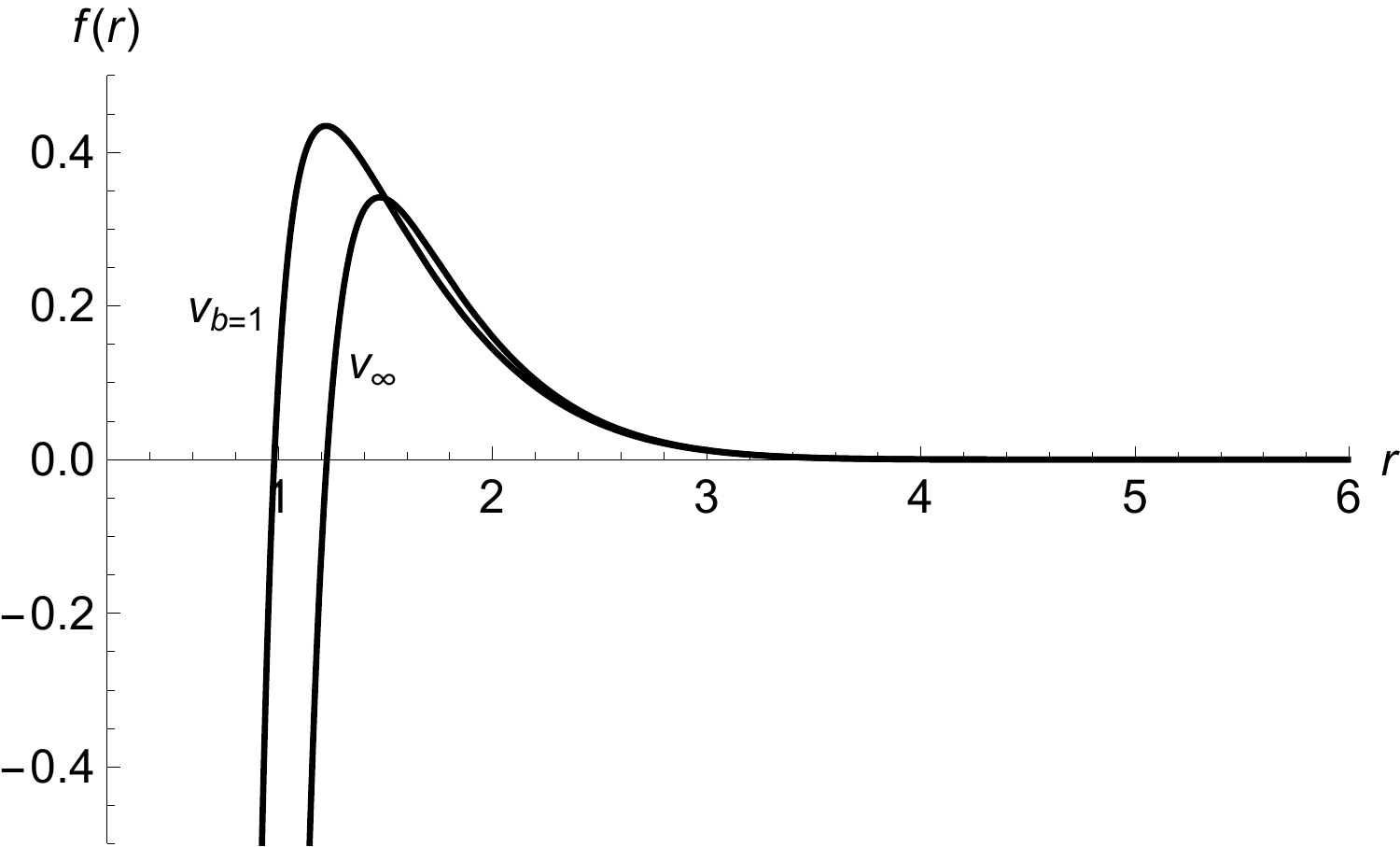}
	\caption{Graph of the uniquely normalized solutions $v(r)$ of $\mathcal{L}_b v = 0$ with $b = 1$ and $\mathcal{L}_\infty v=0$ satisfying $v(r) \to 0$ as $r \to \infty$ for $d = 13$.}
	\label{fig:Morse_index_Lb_Linf}
\end{figure}

	\begin{figure}[htp!]
		\centering
		\includegraphics[width=0.75\textwidth]{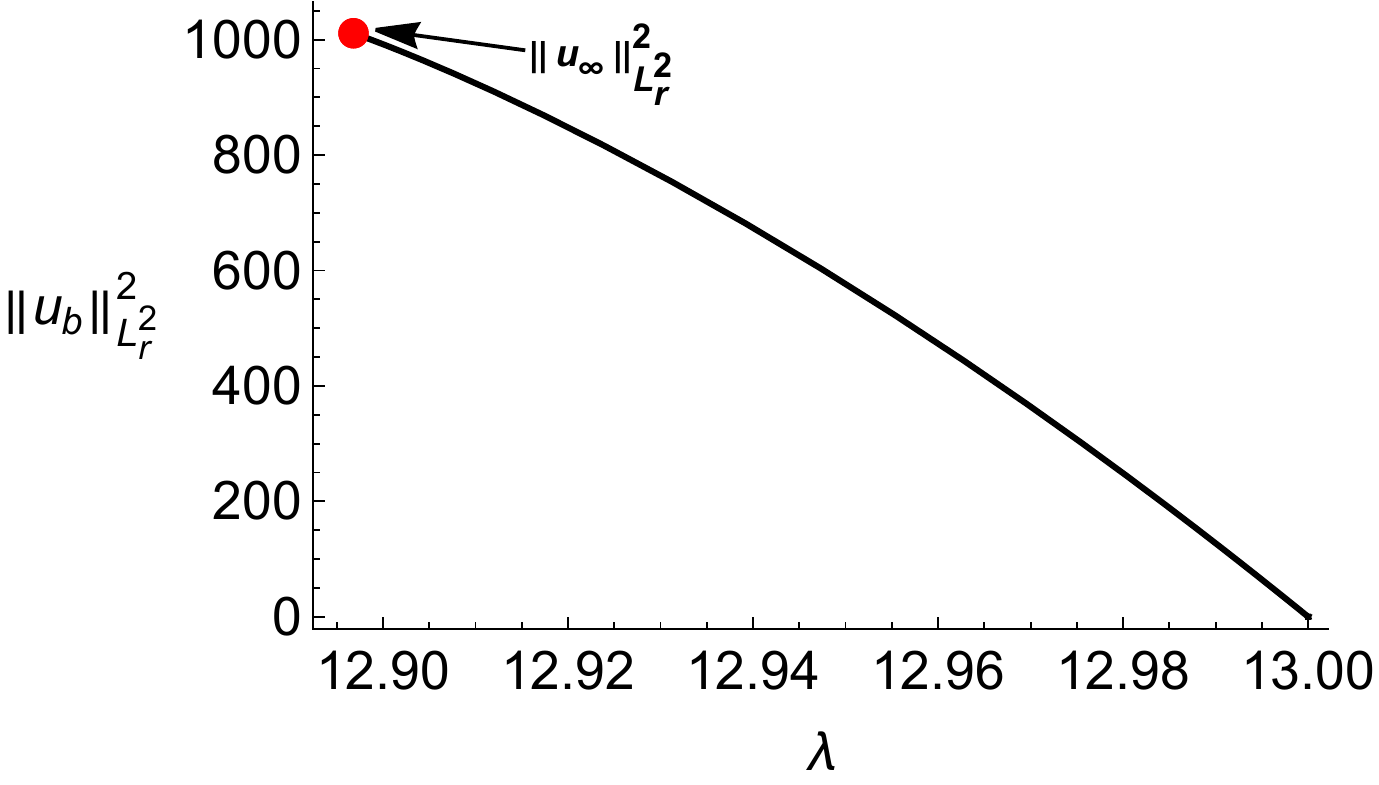}
		\caption{Mass $\| \mathfrak{u}_b \|^2_{L^2_r}$ of the ground state $\mathfrak{u}_b$ for $d=13$ as a function of $\lambda$ together with the mass $\| \mathfrak{u}_{\infty} \|^2_{L^2_r}$ of the limiting singular solution $\mathfrak{u}_\infty$.} 
		\label{fig:mass}
	\end{figure}

By the Vakhitov--Kolokolov stability criterion (Theorem 4.8 in \cite{Pel-book}), if the Morse index is equal to one, the ground state $\mathfrak{u}_b$ is orbitally stable in the time evolution of the Gross--Pitaevskii equation in $\mathcal{E} \cap L^4_r$ for $b \in (b_0,\infty)$ 
if the mapping of $\lambda \mapsto \| \mathfrak{u}_b \|^2_{L^2_r}$ is monotonically decreasing. Figure \ref{fig:mass} shows the dependence of 
the mass $M(\mathfrak{u}_b) = \| \mathfrak{u}_b \|^2_{L^2_r}$ versus $\lambda$ for $\lambda = \lambda(b)$. The red dot depicts the finite value of 
the limiting mass $M(\mathfrak{u}_{\infty}) = \| \mathfrak{u}_{\infty} \|^2_{L^2_r}$. Since the mapping is monotonically decreasing, the Vakhitov--Kolokolov stability criterion asserts that the ground state $\mathfrak{u}_b$ is orbitally stable for $d = 13$ if the time evolution of the Gross--Pitaevskii equation is locally well-posed in $\mathcal{E} \cap L^{4}_r$.

Note that the same conclusion holds for other values of $d$ in the monotone case $d \geq 13$. We have also checked other values of $b$ and 
found no points of bifurcations along the solution family $\lambda(b)$ where $\mathcal{L}_b$ admits zero eigenvalue  
in $L^2_r$. This suggests that the monotone dependence of $\lambda(b)$ with no 
critical points, where $\lambda'(b)$ vanishes, implies no bifurcation points. 
This useful property has not been proven in the literature.
	
	The article is organized as follows. Section \ref{sec-prel} gives 
	the review of preliminary results. Sections \ref{sec-der-b} and \ref{sec-der-c} contain 
	estimates of the derivatives of the two families of solutions 
	described above. The proof of Theorems \ref{theorem-main} and \ref{theorem-main-infty} is described in Section \ref{sec-proofs}.

	\section{Preliminary results}
	\label{sec-prel}
	
	We begin by introducing the Emden-Fowler transformation 
	\begin{equation}
		\label{EF}
		r=e^t, \quad F(r)=\Psi(t), \quad F'(r)=e^{-t}\Psi'(t),
	\end{equation}
	which transforms the differential equation 
	\begin{equation}
	F''(r) + \frac{d-3}{r} F'(r) - \frac{d-3}{r^2} F(r) - r^2 F(r) + \lambda F(r) + \frac{1}{r^2} F(r)^3 = 0,  \quad r > 0
	\end{equation}
	 to the equivalent form
	\begin{equation}
		\label{eq-psi-singular}
		\Psi''(t) + (d-4) \Psi'(t) + (3-d) \Psi(t) + \Psi(t)^3
		= -\lambda e^{2t} \Psi(t) + e^{4t} \Psi(t), \quad t \in \mathbb{R}.
	\end{equation}
	For fixed $d \geq 5$ and $\lambda \in (0,d)$, two one-parameter families of solutions to the  second-order differential equation \eqref{eq-psi-singular} have been constructed in \cite{BIZON2021112358} according to their asymptotic behaviors as $t \to -\infty$ and $t \to +\infty$, respectively.
	
	The first family of solutions to the differential equation (\ref{eq-psi-singular}), denoted as $\left\{\Psi_b\right\}_{b\in \mathbb{R}}$, corresponds to solutions of the initial-value problem \eqref{statGPshoot} after applying the transformation $\Psi_b(t)=e^t f_b(e^t)$. By Lemmas 3.2 and 3.4 in \cite{BIZON2021112358}, 
	$\Psi_b \in C^2(\mathbb{R})$ satisfies the asymptotic behavior
	\begin{equation}
		\label{sol-b-asymptotics}
		\Psi_b(t)=be^t-(\lambda b + b^3)(2d)^{-1}e^{3t}+\mathcal{O}(e^{5t}), \quad \text{as } t\to-\infty,
	\end{equation}
	where the expansion can be differentiated in $t$.
	These solutions depend on $\lambda$ as well, and for $\lambda=\lambda(b)$ and $b > 0$, $\Psi_b(t)$  gives a solution to the boundary-value problem \eqref{statGP}, after the transformation $\mathfrak{u}_b(r) = r^{-1} \Psi_b(\log r)$. For other values of $\lambda$, $\Psi_b(t)$ generally diverges as $t\to +\infty$.
	
	The second family of solutions to the differential equation \eqref{eq-psi-singular}, denoted as $\left\{\Psi_c\right\}_{c\in\mathbb{R}}$, 
	decays to zero as $t \to +\infty$. By Lemmas 3.3 and 3.4 in \cite{BIZON2021112358}, 
	$\Psi_c \in C^2(\mathbb{R})$ satisfies the asymptotic behavior
	\begin{equation}
		\label{sol-c-asymptotics}
		\Psi_c(t)\sim c e^{\frac{\lambda-d+2}{2}t}e^{-\frac{1}{2}e^{2t}}, \quad \text{as } t \to +\infty,
	\end{equation}
	where the sign $\sim$ denotes the asymptotic correspondence which can be differentiated in $t$. Each $\Psi_c(t)$ generally diverges as $t\to-\infty$, except when $\lambda=\lambda(b)$ and $c=c(b)$ for some value of $c(b)$ for which it concides with $\Psi_b(t) = e^t \mathfrak{u}_b(e^t)$:
	\begin{equation}
		\label{bound-state}
		\lambda = \lambda(b) : \quad \Psi_b(t)=\Psi_{c(b)}(t), \quad \mbox{for all} \;\; t\in\mathbb{R}.
	\end{equation}
	
	Each family of solutions is differentiable with respect to parameters $\lambda$ and either $b$ or $c$ due to smoothness of the differential equation (\ref{eq-psi-singular}). Their derivatives decay to zero as $t \to -\infty$ and $t \to +\infty$ respectively, but generally diverge at the other infinities. 
	
	Let us define linearizations of the second-order equation (\ref{eq-psi-singular}) at the two families of solutions: 
	\begin{eqnarray}
		\label{eq:Lb_operator}
		\mathcal{M}_b &:=& \frac{d^2}{dt^2} + (d-4)\frac{d}{dt} +(3-d)+3\Psi_b^2 +\lambda e^{2t} -e^{4t}, \\
		\mathcal{M}_c &:=& \frac{d^2}{dt^2} + (d-4)\frac{d}{dt} + (3-d) + 3\Psi_c^2 + \lambda e^{2t} -e^{4t}.
		\label{eq:Lc_operator}
	\end{eqnarray}
	Then, differentiating the second-order equation \eqref{eq-psi-singular} with respect to $b$ and $c$ at fixed $\lambda$ yields 
	\begin{align}
		\label{der-solutions}
		\mathcal{M}_b \partial_b\Psi_b = 0, \quad \mathcal{M}_c \partial_c \Psi_c = 0,
	\end{align}
	where $\partial_b \Psi_b(t) \to 0$ as $t \to -\infty$ and 
	$\partial_c \Psi_c(t) \to 0$ as $t \to +\infty$.
	
	The first family $\left\{\Psi_b\right\}_{b\in \mathbb{R}}$ is defined 
	in a neighborhood of a heteroclinic orbit $\Theta$ connecting the saddle point $(0,0)$ and the stable point $(\sqrt{d-3},0)$ of the truncated autonomous version of equation \eqref{eq-psi-singular} given by 
	\begin{equation}
		\Theta''(t)+(d-4)\Theta'(t)+(3-d)\Theta(t)+\Theta(t)^3=0.
		\label{eq:theta_nonlinear}
	\end{equation}
	By Lemma 6.1 in \cite{BIZON2021112358}, there exists a heteroclinic orbit between $(0,0)$ and $(\sqrt{d-3},0)$ 
	which is defined uniquely (module to the translation in $t$) by the asymptotic behavior 
	\begin{equation}
		\label{sol-Theta-asymptotics-1}
		\Theta(t) = e^t-(2d)^{-1}e^{3t}+\mathcal{O}(e^{5t}), \quad \text{as } t\to-\infty.
	\end{equation}
	The following proposition presents the main result of Lemmas 6.2, 6.5, and 6.8 in \cite{BIZON2021112358}.
	
	\begin{proposition}
		\label{prop-b-solution}
		Fix $d \geq 5$ and $\lambda \in \mathbb{R}$. For every $T > 0$ and $a \in (0,1)$, there exist $(T,a)$-independent constants $b_0 > 0$, $C_0 > 0$ 
		and $(T,a)$-dependent constants $b_{T,a} > 0$, $C_{T,a} > 0$ such that the unique solution $\Psi_b$ to the differential equation (\ref{eq-psi-singular}) with the asymptotic behavior (\ref{sol-b-asymptotics}) satisfies for every $b \in (b_0,\infty)$
		\begin{equation}
			\label{bound-solution-b}
			\sup_{t \in (-\infty,0]} |\Psi_b(t-\log b) - \Theta(t)| \leq C_0 b^{-2} e^{3t}
		\end{equation}
		and  for every $b \in (b_{T,a},\infty)$
		\begin{equation}
			\label{bound-solution-b-extended}
			\sup_{t \in [0,T + a \log b ]} |\Psi_b(t-\log b) - \Theta(t)| \leq C_{T,a} b^{-2(1-a)}.
		\end{equation}
	\end{proposition}
	
	The heteroclinic orbit of the truncated equation (\ref{eq:theta_nonlinear}) connects the saddle point $(0,0)$ 
	associated with the characteristic exponents $\kappa_1 =1$ and $\kappa_2=3-d$ 
	and the stable point $(\sqrt{d-3},0)$ associated with the characteristic 
	exponentis $\kappa_+$ and $\kappa_-$ given by 
	\begin{equation}
		\label{kappa-pm}
		\kappa_\pm = -\frac{1}{2}(d-4)\pm \frac{1}{2}\sqrt{d^2-16d+40}.
	\end{equation}
	For $d\geq 13$, the characteristic exponents are real and satisfy 
	$\kappa_-<\kappa_+<0$. We make the following assumption 
	on how the heteroclinic orbit converges to the stable point $(\sqrt{d-3},0)$.
	
	\begin{assumption}
		\label{assumption-1} 
		Assume that there exists $A_0 \neq 0$ such that 
		\begin{equation}
			\label{sol-Theta-asymptotics-2}
			\Theta(t) = \sqrt{d-3} + A_0 e^{\kappa_+t} + \mathcal{O}(e^{\kappa_- t},e^{2\kappa_+ t}) \quad \mbox{\rm as} \quad t \to +\infty.
		\end{equation}
	\end{assumption}
	
	\begin{remark}
		Assumption \ref{assumption-1} implies that $\Theta(t)$ converges to $\sqrt{d-3}$ as $t \to +\infty$ according to the slowest decay rate given by $\kappa_+$. It is not a priori clear why the constant $A_0$ could not be zero in exceptional cases, for which $\Theta(t)$ converges to $\sqrt{d-3}$ as $t \to +\infty$ according to the fastest decay rate given by $\kappa_-$.
	\end{remark}
	
	The second family $\left\{\Psi_c\right\}_{c \in \mathbb{R}}$ is defined 
	in a neighborhood of the special solution $\Psi_{\infty}(t) := e^t \mathfrak{u}_{\infty}(e^t)$ obtained from the limiting singular solution $\mathfrak{u}_{\infty} \in \mathcal{E}\cap L^4_r$. This special solution corresponds to 
	the values of $\lambda=\lambda_\infty$ and $c=c_\infty$ so that 
	\begin{equation}
		\label{lim-singular-solution}
		\lambda = \lambda_{\infty} : \quad \Psi_{\infty}(t) = \Psi_{c_{\infty}}(t), \quad 
		\mbox{\rm for all \;} t \in \mathbb{R}.
	\end{equation}
The solution $\Psi_{\infty}$ satisfies the asymptotic behaviors
	\begin{equation}
		\label{sol-Psi-infty-asymptotics}
		\Psi_\infty(t) = \sqrt{d-3}\left[1-\frac{\lambda_\infty}{4d-10}e^{2t}+ \mathcal{O}(e^{4t})\right], \quad \text{as } t\to-\infty
	\end{equation}
and
	\begin{equation}
\label{sol-c-asymptotics-infty}
\Psi_{\infty}(t) \sim c_{\infty} e^{\frac{\lambda_{\infty}-d+2}{2}t}e^{-\frac{1}{2}e^{2t}}, \quad \text{as } t \to +\infty,
\end{equation}
	The following proposition presents a modification of Lemmas 6.6 and 6.9 in \cite{BIZON2021112358}. Since the statement was not proven
	in \cite{BIZON2021112358}, we give the precise proof of this result 
	in Appendix A.

	\begin{proposition}
		\label{lemma:Psi_c_neg_t}
		Fix $d\geq 13$ and $a\in(0,1)$. There exist constants $b_0>0$, $C_0>0$, and $\epsilon_0>0$, such that for every $\epsilon\in(0,\epsilon_0)$, $b \in (b_0,\infty)$, and $(\lambda,c)\in\mathbb{R}^2$ satisfying
		\begin{equation}
			\label{lambda-c}
			|\lambda-\lambda_{\infty}|+|c-c_\infty|\leq \epsilon b^{\kappa_-(1-a)},
		\end{equation}
		it is true for every $t\in[(a-1)\log b,0]$ that
		\begin{equation}
			|\Psi_c(t)-\Psi_\infty(t)|\leq C_0 \epsilon b^{\kappa_-(1-a)}e^{\kappa_-t}.
			\label{eq:Psi_c_bound_neg_t}
		\end{equation}
	\end{proposition}
	
	\begin{remark}
		Note that the divergent behavior of $e^{\kappa_-t}$ for large negative $t$ in \eqref{eq:Psi_c_bound_neg_t} is cancelled by the decay of $b^{\kappa_-(1-a)}$ on any fixed interval $[(a-1)\log b,0]$. Thus, bound \eqref{eq:Psi_c_bound_neg_t} implies 
		\begin{equation}
			\label{uniform-bound}
			\sup_{t\in[(a-1)\log b,0]}|\Psi_c(t)-\Psi_\infty(t)|\leq C_0\epsilon,
		\end{equation}
		for every $(\lambda,c) \in \mathbb{R}^2$ satisfying (\ref{lambda-c}).
	\end{remark}

\begin{remark}
Since $\Psi_c$ is smooth in $\lambda$ and $c$ and has the same decay 
(\ref{sol-c-asymptotics}) as $t \to +\infty$ in comparison with 
(\ref{sol-c-asymptotics-infty}) for $\Psi_{\infty}$, it is true for every $(\lambda,c)$ in a local neighborhood of $(\lambda_{\infty},c_{\infty})$ that 
	\begin{equation}
	\sup_{t \in [0,\infty)}|\Psi_c(t)-\Psi_\infty(t)|\leq C_0 (|\lambda - \lambda_{\infty}| + |c - c_{\infty}|).
	\label{eq:Psi_c_bound_pos_t}
	\end{equation}
\end{remark}
	
	Linearization of the second-order equation (\ref{eq-psi-singular}) at $\Psi_{\infty}$ is given by 
	\begin{equation}
		\mathcal{M}_{\infty} := \frac{d^2}{dt^2} + (d-4)\frac{d}{dt} +(3-d)+3\Psi_{\infty}^2 +\lambda_{\infty} e^{2t} -e^{4t}.
		\label{eq:L_infty_op}
	\end{equation}	
	Differentiating the second-order equation (\ref{eq-psi-singular}) with respect to $c$ at fixed $\lambda$ and then substituting $c = c_{\infty}$ and $\lambda = \lambda_{\infty}$ gives 
	\begin{align}
		\label{der-solution-infty}
		\mathcal{M}_{\infty} \partial_c \Psi_{\infty} = 0,
	\end{align}
	where $\partial_c\Psi_\infty$ is a short notation for  $\partial_c\Psi_c|_{(\lambda,c)=(\lambda_\infty,c_\infty)}$.
	The function  $\partial_c\Psi_\infty(t)$ decays fast as $t \to +\infty$ 
	according to (\ref{sol-c-asymptotics}), but generally diverges as $t \to -\infty$. 
	Since $\Psi_{\infty}(t) \to \sqrt{d-3}$ as $t \to -\infty$, the divergence of $\partial_c\Psi_\infty(t)$ as $t \to -\infty$ is defined by the same 
	two characteristic exponents $\kappa_+$ and $\kappa_-$ given by (\ref{kappa-pm}). We make the following assumption on the divergence of this solution. 
	
	\begin{assumption}
		\label{assumption-2}
		Assume that there exists $L_{\infty} \neq 0$ such that 
		\begin{equation}
			\label{sol-Psi-Infty-der}
			\partial_c\Psi_\infty(t) = L_{\infty} e^{\kappa_- t} + 
			\mathcal{O}(e^{\kappa_+ t},e^{(\kappa_-+2)t}) \quad \mbox{\rm as} \quad t \to -\infty.
		\end{equation}
	\end{assumption}
	
	\begin{remark}
		Assumption \ref{assumption-2} implies that $\partial_c\Psi_\infty(t)$ diverges as $t \to -\infty$ with the fastest growth rate given by $\kappa_-$. Again, it is not a priori clear why the constant $L_{\infty}$ could not be zero in exceptional cases, for which $\partial_c\Psi_\infty(t)$ diverges as $t \to -\infty$ with the slowest growth rate given by $\kappa_+$.
	\end{remark}

\begin{figure}[htp!]
	\centering
	\includegraphics[width=0.72\textwidth]{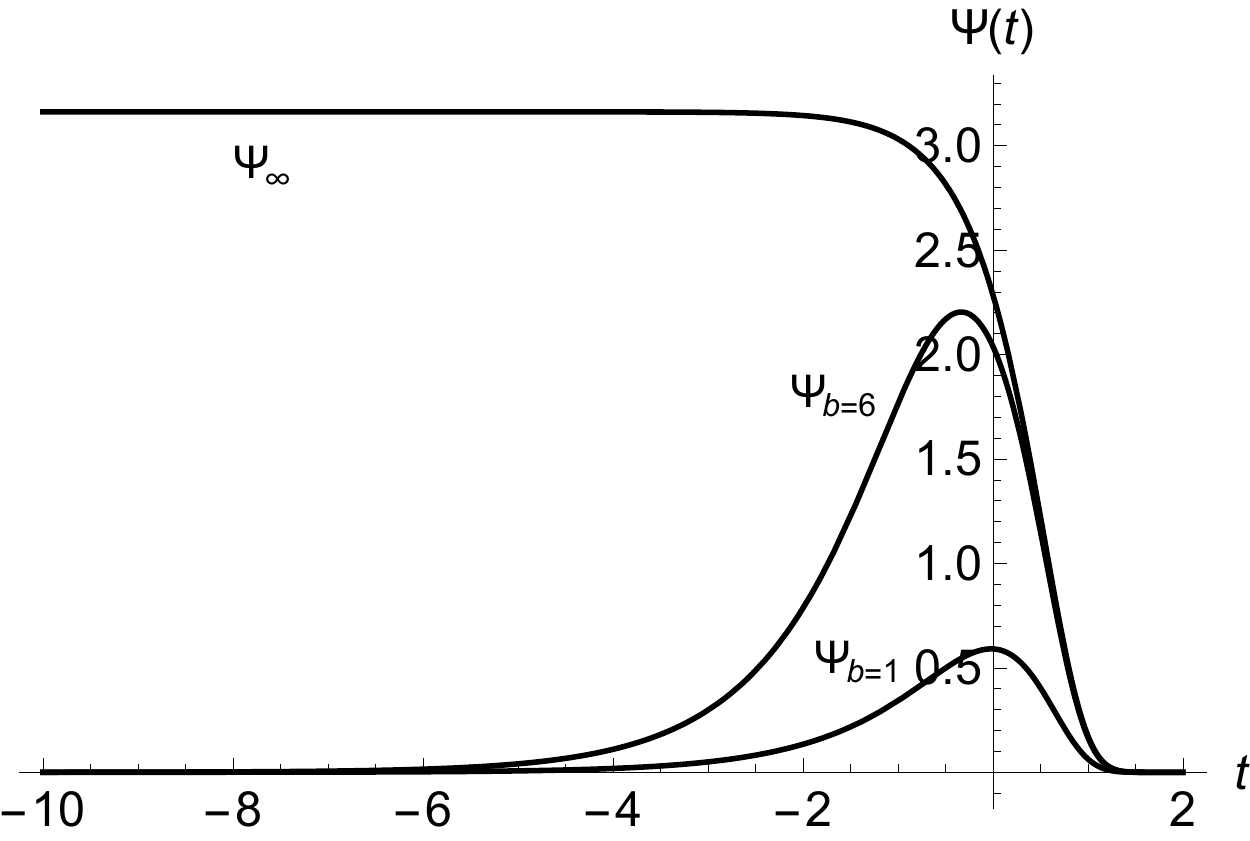}
	\caption{Graph of the  solution $\Psi_b$ for $b = 1$ and $b = 6$ in comparison with the solution $\Psi_\infty$  for $d = 13$.} 
	\label{fig:Psib_Psiinf_b_1_d_13}
\end{figure}

\begin{figure}[htp!]
	\centering
	\includegraphics[width=0.72\textwidth]{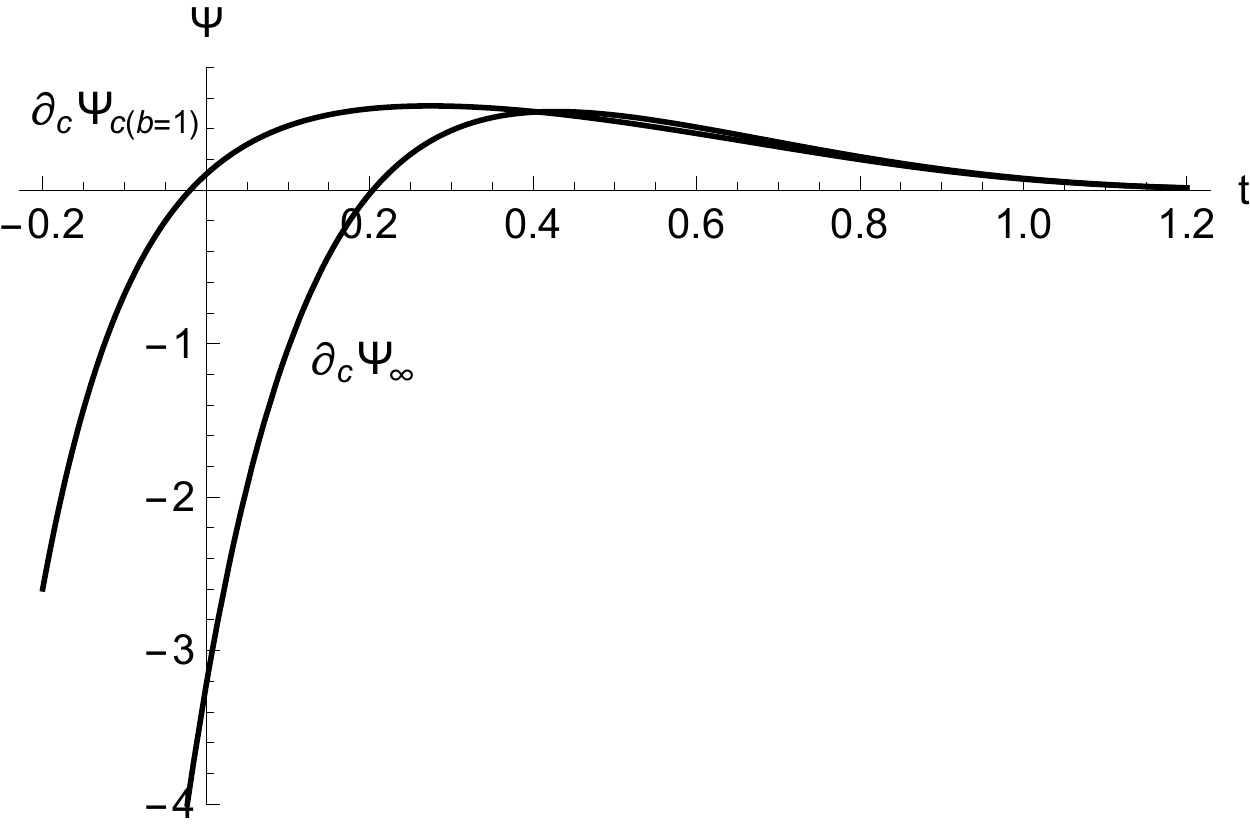}
	\caption{Graph of $\partial_c\Psi_{c(b)}$ with $b = 1$ and $\partial_c\Psi_\infty$ for $d = 13$.}
	\label{fig:Morse_index_Mb_Minf}
\end{figure}	

Figure \ref{fig:Psib_Psiinf_b_1_d_13} shows $\Psi_b(t)$ for two values of $b$ and $\Psi_{\infty}(t)$ for  $d = 13$. After the inverse Emden-Fowler transformation (\ref{EF}) and the transformation $\mathfrak{u}(r) = r^{-1} F(r)$, these functions give $\mathfrak{u}_b(r)$ and $\mathfrak{u}_{\infty}(r)$, as shown on Figure \ref{fig:ub_u_inf_sol}.

Figure \ref{fig:Morse_index_Mb_Minf} shows $\partial_c\Psi_{c(b)}(t)$ with $b = 1$ and $\partial_c\Psi_\infty(t)$ for $d = 13$. These functions are solutions of the homogeneous equaitons $\mathcal{M}_b\partial_c\Psi_{c(b)} = 0$ and $\mathcal{M}_\infty\partial_c\Psi_\infty=0$. After the transformation $v(r) = r^{-1} \partial_c \Psi_{c(b)}(\log r)$, 
these functions give solutions of $\mathcal{L}_b v = 0$ and $\mathcal{L}_{\infty} v = 0$ that decay to zero as $r \to \infty$, as 
shown in Figure \ref{fig:Morse_index_Lb_Linf}. Since $\partial_c\Psi_{c(b)}(t)$ and $\partial_c\Psi_\infty(t)$ have only one zero on $\mathbb{R}$, 
the corresponding functions $v(r)$ have only one zero on $(0,\infty)$, 
so that the Morse index of $\mathcal{L}_b$ and $\mathcal{L}_{\infty}$ is equal to one.

	\section{Derivative of the $b$-family of solutions}
	\label{sec-der-b}
	
	Here we describe the asymptotic behavior of $\partial_b\Psi_b$. The following lemma shows that after translation by $-\log b$, $\partial_b\Psi_b$ converges to $b^{-1} \Theta'$ on the negative $t$-axis. Moreover, the estimate can be extended from $(-\infty,0]$ to $[0,T +a \log b]$ for fixed $T \in \mathbb{R}$ and $a \in (0,1)$ and for sufficiently large $b$ at the expense of slower convergence rate.
	
	\begin{lemma}
		\label{lemma:part_b_Psi_b_est}
		Let $d\geq 13$ and $\lambda\in\mathbb{R}$. For every $T > 0$ and $a \in (0,1)$,  there exist $(T,a)$-independent constants $b_0 > 0$, $C_0 > 0$ 
		and $(T,a)$-dependent constants $b_{T,a} > 0$, $C_{T,a} > 0$ such that 
		$\partial_b \Psi_b$ satisfies for every $b \in (b_0,\infty)$
		\begin{equation}
			\sup_{t\in (-\infty,0]}|\partial_b\Psi_b(t-\log b)-b^{-1}\Theta'(t)|+\sup_{t\in (-\infty,0]}|\partial_b\Psi_b'(t-\log b)-b^{-1}\Theta''(t)|\leq C_0b^{-3}
			\label{eq:part_b_Psi_b_est_neg_t}
		\end{equation}
		and for every $b \in (b_{T,a},\infty)$
		\begin{equation}
			\sup_{t\in\left[0,T+a\log b\right]}|\partial_b\Psi_b(t-\log b)-b^{-1}\Theta'(t)|+\sup_{t\in\left[0,T+a\log b\right]}|\partial_b\Psi_b'(t-\log b)-b^{-1}\Theta''(t)|\leq C_{T,a} b^{-2(1-a)-1}.
			\label{eq:part_b_Psi_b_est_intermediate}
		\end{equation}
	\end{lemma}
	
	\begin{proof}
		We begin by introducing $\gamma(t) := \partial_b\Psi_b(t-\log b) - b^{-1}\Theta'(t)$, where $\Psi_b$ is the unique solution to (\ref{eq-psi-singular}) with the asymptotic behavior (\ref{sol-b-asymptotics}) and $\Theta$ is the unique solution 
		to (\ref{eq:theta_nonlinear}) with the asymptotic behavior 
		(\ref{sol-Theta-asymptotics-1}). Since  $\partial_b\Psi_b$ satisfies $\mathcal{M}_b \partial_b \Psi_b = 0$ 
		and $\Theta'$ satisfies $\mathcal{M}_0 \Theta' = 0$, where 
		$\mathcal{M}_b$ is given by (\ref{eq:Lb_operator}) and 
		\begin{equation}
			\mathcal{M}_0 := \frac{d^2}{dt^2} + (d-4) \frac{d}{dt} + (3-d) + 3 \Theta^2,
			\label{Theta-prime}
		\end{equation}
		the difference term $\gamma(t)$ satisfies the following equation:
		\begin{equation}
			\mathcal{M}_0 \gamma = f_b(b^{-1}\Theta'+\gamma),
			\label{eq:pers_part_b_Psi_b}
		\end{equation}
		where
		\begin{align}
			f_b(t) :=3(\Theta(t)^2-\Psi_b(t-\log b)^2)-\lambda b^{-2}e^{2t}+b^{-4}e^{4t}.
			\label{eq:f_b_source}
		\end{align}
		Note that $\mathcal{M}_0 - f_b$ gives $\mathcal{M}_b$ after translation $t \mapsto t + \log b$. \\

		{\em Proof of the bound (\ref{eq:part_b_Psi_b_est_neg_t}).} Two linearly independent solutions of $\mathcal{M}_0 \gamma = 0$ are given by $\Theta'(t)$ and another function $\Xi(t)$, which can be found from the Wronskian relation
		\begin{equation}
			W(\Theta',\Xi)(t)=\Theta'(t)\Xi'(t)-\Theta''(t)\Xi(t) = W_0e^{(4-d)t},
			\label{eq:Wronskian_Teta'_Xi}
		\end{equation}
		for some constant $W_0\neq 0$. We take $W_0 = 1$ in order to normalize $\Xi(t)$ uniquely. Since $\Theta'(t) \to 0$ as $t\to-\infty$ according to the asymptotic expansion
		\begin{equation}
			\Theta'(t) = e^t + \mathcal{O}(e^{3t}), \quad \text{as } t\to-\infty,
			\label{eq:Theta_assymptotics_neg_t}
		\end{equation}
		we have $\Xi(t) \to \infty$ as $t \to -\infty$ according to 
		the asymptotic expansion
		\begin{equation}
			\Xi(t) = (2-d)^{-1}e^{(3-d)t} + \mathcal{O}(e^{(5-d)t}), \quad \text{as } t\to-\infty.
			\label{eq:Xi_assymptotics_neg_t}
		\end{equation}
		
		In order to estimate the supremum-norm of $\gamma(t)$ for $t\in(-\infty,0]$, we first rewrite the differential equation \eqref{eq:pers_part_b_Psi_b} as an integral equation
		\begin{equation}
			\gamma(t) = \int_{-\infty}^t e^{(d-4)t'}[\Theta'(t')\Xi(t)-\Theta'(t)\Xi(t')]f_b(t')[b^{-1}\Theta'(t')+\gamma(t')]dt',
			\label{eq:gamma_int_-infty_to_t}
		\end{equation}
		where the free solution $c_1 \Theta'(t) + c_2 \Xi(t)$ is set to zero from the requirement that $\gamma(t) = \mathcal{O}(e^{3t})$ as $t \to -\infty$. The integral kernel in \eqref{eq:gamma_int_-infty_to_t} becomes bounded if we introduce the transformation $\tilde{\gamma}(t)=e^{-t}\gamma(t)$. The integral equation corresponding to $\tilde{\gamma}$ is
		\begin{equation}
			\tilde{\gamma}(t)=\int_{-\infty}^t K_1(t,t') f_b(t')[b^{-1}e^{-t'}\Theta'(t')+\tilde{\gamma}(t')]dt',
			\label{eq:gamma_tilde_int_-infty_to_t}
		\end{equation}
		where
		\begin{equation}
			K_1(t,t'):=[e^{-t'}\Theta'(t')][e^{(d-3)t}\Xi(t)]e^{(d-2)(t'-t)}-[e^{-t}\Theta'(t)][e^{(d-3)t'}\Xi(t')].
			\label{eq:K1_est}
		\end{equation}
		Since $\Theta'(t) = \mathcal{O}(e^t)$ and $\Xi(t) = \mathcal{O}(e^{(3-d)t})$ as $t \to -\infty$, the integral kernel  $K_1(t,t')$ is a bounded function for all $t,t'\in(-\infty,0]$.
		
		By using bound (\ref{bound-solution-b}) of Proposition \ref{prop-b-solution}, we obtain from (\ref{eq:f_b_source}) that 
		\begin{equation}
			|f_b(t)|\leq C_0 b^{-2}e^{2t}, \quad t\in(-\infty,0],
			\label{eq:f_b_est}
		\end{equation}
		where the constant $C_0$ is independent of $b$ for sufficiently large $b$ and may change from one line to another line. Boundedness of $K_1$ in the integral equation (\ref{eq:gamma_tilde_int_-infty_to_t}) on $(-\infty,0] \times (-\infty,0]$ and the estimate \eqref{eq:f_b_est} allow us to estimate the supremum norm of $\tilde{\gamma}(t)$ on $(-\infty,0]$ as follows
		\begin{equation}
			\|\tilde{\gamma}\|_\infty \leq C_0 b^{-2} (b^{-1} \|e^{-t}\Theta'\|_\infty+ \|\tilde{\gamma}\|_\infty).
		\end{equation}
		Due to smallness of $b^{-2}$ and boundness of $b$-independent 
		$\|e^{-t}\Theta'\|_\infty$, this estimate implies that 
		\begin{equation}
			\sup_{t\in(-\infty,0]}|\tilde{\gamma}(t)| \leq C_0 b^{-3}.
			\label{eq:gamma_tilde_bound_neg_t}
		\end{equation}
		Since $|\gamma(t)|\leq |\tilde{\gamma}(t)|$ for all $t\in(-\infty,0]$,
		we obtain the first part of bound  \eqref{eq:part_b_Psi_b_est_neg_t}. Since $\tilde{\gamma}\in C^1(-\infty,0)$,  we obtain the second part of bound \eqref{eq:part_b_Psi_b_est_neg_t} by differentiting equation \eqref{eq:gamma_tilde_int_-infty_to_t} and using \eqref{eq:gamma_tilde_bound_neg_t}.\\
		
		{\em Proof of the bound \eqref{eq:part_b_Psi_b_est_intermediate}.} 
		In order to estimate $|\gamma(t)|$ for sufficiently large positive $t$, 
		we need to define solutions to $\mathcal{M}_0 \gamma = 0$ from their behavior as $t\to +\infty$. Since $(\sqrt{d-3},0)$ is a stable node of  the nonlinear equation (\ref{eq:theta_nonlinear}), we can pick two linearly independent solutions $\gamma_1(t)$ and $\gamma_2(t)$ 
		from their decaying behavior
		\begin{equation}
			\gamma_1(t)  =\mathcal{O}(e^{\kappa_-t}), \quad  \gamma_2(t)=\mathcal{O}(e^{\kappa_+t}) \quad \mbox{\rm as} \;\; t\to +\infty,
			\label{eq:gamma_1_2_est_pos_t}
		\end{equation}
		where $\kappa_- < \kappa_+<0$ are given by (\ref{kappa-pm}).
		The Liouville's formula yields the Wronskian relation
		\begin{equation}
			W(\gamma_1,\gamma_2)(t)=\gamma_1(t)\gamma_2'(t)-\gamma_1'(t)\gamma_2(t) = W_0e^{(4-d)t},
		\end{equation}
		for some constant $W_0\neq 0$, and by normalizing $\gamma_1(t)$ 
		and $\gamma_2(t)$ we can assume that $W_0=1$. In order to derive supremum-norm estimates for $\gamma(t)$, we once again rewrite differential equation \eqref{eq:pers_part_b_Psi_b} as an integral equation
		\begin{multline}
			\gamma(t) = \gamma(0)[\gamma_1(t)\gamma_2'(0)-\gamma_1'(0)\gamma_2(t)]+\gamma'(0)[\gamma_1(0)\gamma_2(t)-\gamma_1(t)\gamma_2(0)]\\+\int_0^t e^{(d-4)t'}[\gamma_1(t')\gamma_2(t)-\gamma_1(t)\gamma_2(t')]f_b(t')[b^{-1}\Theta'(t')+\gamma(t')]dt',
			\label{eq:gamma_int_0_to_t}
		\end{multline}
		this time for $t\in[0,T+a\log b]$. From bound \eqref{eq:part_b_Psi_b_est_neg_t} we obtain existence of a constant $C_0>0$ and $b_0>0$, such that
		\begin{equation}
			|\gamma(0)|+|\gamma'(0)|\leq C_0b^{-3}, \quad \text{for all } b\geq b_0.
			\label{eq:gamma_0_est}
		\end{equation}		
		Due to the decay of $\gamma_1(t)$ and $\gamma_2(t)$ as $t \to +\infty$, the kernel of the integral equation \eqref{eq:gamma_int_0_to_t} behaves like $e^{\kappa_+(t-t')}$ and $e^{\kappa_-(t-t')}$, and is thus bounded as $t\to +\infty$ since $t' \leq t$. 	By using bound (\ref{bound-solution-b-extended}) of Proposition \ref{prop-b-solution}, we obtain from (\ref{eq:f_b_source}) that
		\begin{equation}
			\sup_{t\in[0,T+a\log b]}|f_b(t)|\leq C_{T,a} b^{-2(1-a)},
			\label{eq:f_b_est_intermediate}
		\end{equation}
		where the $b$-independent constant $C_{T,a}$  may change from one line to another line. Using estimates \eqref{eq:gamma_1_2_est_pos_t},  \eqref{eq:gamma_0_est},  and \eqref{eq:f_b_est_intermediate}, 
		we obtain from the integral equation \eqref{eq:gamma_int_0_to_t} the following bound on the supremum-norm of $\gamma(t)$ on $[0,T+a\log b]$:
		\begin{equation}
			\|\gamma\|_\infty \leq C_{T,a}  \left( b^{-3} + b^{-2(1-a)-1} \| \Theta'\|_{L^1} + (T+a\log b) b^{-2(1-a)} \|\gamma \|_\infty\right).
		\end{equation}
		Due to smallness of $(T+a\log b) b^{-2(1-a)}$ and boundedness of $b$-independent $\| \Theta'\|_{L^1}$, this estimate implies that
		\begin{equation}
			\sup_{t\in[0,T+a\log b]}|\gamma(t)| \leq C_{T,a} b^{-2(1-a)-1}.
			\label{eq:gamma_est_intermediate}
		\end{equation}
		By differentiting equation \eqref{eq:gamma_int_0_to_t} and using \eqref{eq:gamma_est_intermediate}, we obtain a similar bound on $\gamma'(t)$, which together with \eqref{eq:gamma_est_intermediate} gives us bound \eqref{eq:part_b_Psi_b_est_intermediate}.
	\end{proof}
	
	Bound (\ref{eq:part_b_Psi_b_est_intermediate}) of Lemma \ref{lemma:part_b_Psi_b_est} and the expansion (\ref{sol-Theta-asymptotics-2}) imply the following important 
	representation of $\partial_b \Psi_b(t)$ at $t = T + (a-1) \log b$.
	
	\begin{corollary}
		Under Assumption \ref{assumption-1}, there exist some constant $a_0 \in (0,0.5)$  such that for every $a \in (0,a_0)$ and $T > 0$, there exist some $(T,a)$-dependent constants $b_{T,a} > 0$ and $C_{T,a} > 0$ such that for every $b \in (b_{T,a},\infty)$ we have 
		\begin{equation}
			|\partial_b\Psi_b(T+(a-1)\log b)-A_0\kappa_+b^{a\kappa_+-1} e^{\kappa_+T}|\leq C_{T,a} \max\left\{b^{-2(1-a)-1},b^{a\kappa_--1},b^{2a\kappa_+-1}\right\}.
			\label{eq:part_b_Psi_b_bound_evaluated}
		\end{equation}
		\label{corollary-1}
	\end{corollary}
	
	\begin{proof}	
		Bound \eqref{eq:part_b_Psi_b_bound_evaluated} follows from the expansion \eqref{sol-Theta-asymptotics-2} at large positive $t = T + a \log b$, the bound  \eqref{eq:part_b_Psi_b_est_intermediate} at $t = T + (a-1) \log b$, and the triangle inequality if 
		$$
		b^{a\kappa_+-1}\gg b^{-2(1-a)-1}.
		$$ 
		This constraint is satisfied if $a \in (0,a_0)$, where
		\begin{equation}
			\label{a-0}
			a_0:=\frac{2}{2+|\kappa_+|}=\frac{4}{d-\sqrt{d^2-16d+40}}=\frac{d+\sqrt{d^2-16d+40}}{2(2d-5)}.
		\end{equation}
		Note that $a_0 \in (0,0.5)$ for every $d \geq 13$.
	\end{proof}
	
	\begin{remark}
		Lemma \ref{lemma:part_b_Psi_b_est} suggests that both bounds (\ref{bound-solution-b}) and (\ref{bound-solution-b-extended}) of Proposition \ref{prop-b-solution} can be differentiated in $b$ after translation: $t \mapsto t + \log b$. This $C^1$ property would also follow 
		from applications of Banach fixed-point theorem in \cite{BIZON2021112358} due to contraction of integral operators and smoothness of the vector fields. Since the $C^1$ property was not written explicitly in \cite{BIZON2021112358}, we provided the precise proof of Lemma \ref{lemma:part_b_Psi_b_est}.
	\end{remark}
	
	\section{Derivative of the $c$-family of solutions}
	\label{sec-der-c}
	
	Here we describe the asymptotic behavior of $\partial_c\Psi_c$. 
	Since $\Psi_c$ is smooth in $\lambda$ and $c$ and has the same decay as $t \to +\infty$ as the limiting solution $\Psi_{\infty}$ according to (\ref{sol-c-asymptotics}) and (\ref{sol-c-asymptotics-infty}), $\partial_c \Psi_c$ converges to $\partial_c \Psi_{\infty}$ on $[0,\infty)$. To be precise, there exists a constant $C_0 > 0$ such that for every $(\lambda,c)$ in a local neighborhood of $(\lambda_{\infty},c_{\infty})$, we have 
	\begin{equation}
		\label{bound-on-positive}
		\sup_{t \in [0,\infty)} |\partial_c \Psi_c(t) - \partial_c \Psi_{\infty}(t)| + \sup_{t \in [0,\infty)} |\partial_c \Psi_c'(t) - \partial_c \Psi_{\infty}'(t)| \leq C_0 (|\lambda - \lambda_{\infty}| + |c - c_{\infty}|).
	\end{equation}
	The following lemma extends the estimate on the difference $|\partial_c\Psi_c(t)-\partial_c\Psi_\infty(t)|$ from $[0,\infty)$ to 
	$[(a-1)\log b,0]$ for fixed $a \in (0,1)$ and for sufficiently large $b$
	provided that $(\lambda,c)$ are sufficiently close to $(\lambda_\infty,c_\infty)$.

	\begin{lemma}
		\label{lemma:part_c_Psi_c_est}
		Fix $d\geq 13$. For fixed $a\in(0,1)$, there exist $b_0>0$, $C_0>0$, and $\epsilon_0>0$, such that for every $\epsilon\in(0,\epsilon_0)$ and for every $(\lambda,c)\in\mathbb{R}^2$ satisfying
		\begin{equation}
			|\lambda-\lambda_{\infty}|+|c-c_\infty|\leq \epsilon b^{\kappa_-(1-a)},
		\end{equation}
		it is true for every $b\geq b_0$ and every $t\in[(a-1)\log b,0]$ that
		\begin{equation}
			|\partial_c\Psi_c(t)-\partial_c\Psi_\infty(t)|+|\partial_c\Psi_c'(t)-\partial_c\Psi_\infty'(t)| \leq C_0 \epsilon e^{\kappa_- t}.
			\label{eq:r_est_intermediate}
		\end{equation}
	\end{lemma}
	
	\begin{proof}
		Let $r(t) := \partial_c\Psi_c(t) - \partial_c\Psi_\infty(t)$.
		Since $\partial_c \Psi_c$ satisfies $\mathcal{M}_c \partial_c \Psi_c = 0$ and $\partial_c \Psi_{\infty}$ satisfies $\mathcal{M}_{\infty} \partial_c \Psi_{\infty} = 0$, the difference term $r$ satisfies the following equation:
		\begin{equation}
			\label{eq-for-r}
			\mathcal{M}_\infty r= f_c +g_c r,
		\end{equation}
		where 
		\begin{align}
			f_c(t)&:=3(\Psi_\infty(t)^2-\Psi_c(t)^2)\partial_c\Psi_\infty(t)-(\lambda-\lambda_\infty)e^{2t}\partial_c\Psi_\infty(t),\\
			g_c(t)&:=3(\Psi_\infty(t)^2-\Psi_c(t)^2)+(\lambda_\infty-\lambda)e^{2t}.
		\end{align}
		Note that $\mathcal{M}_c = \mathcal{M}_{\infty} - g_c$.
		
		As in Appendix \ref{appendix-a}, we pick two linearly independent solutions $r_1$, $r_2$ to $\mathcal{M}_\infty r = 0$ such that 
		\begin{equation}
			\label{r1-r2-main-text}
			r_1(t) = \mathcal{O}(e^{\kappa_-t}), \quad r_2(t)=\mathcal{O}(e^{\kappa_+t}), \quad \text{as } t\to-\infty,
		\end{equation}
		where $\kappa_- < \kappa_+< 0$ are given by (\ref{kappa-pm}). Using the method of variation of parameters, we rewrite the differential equation (\ref{eq-for-r}) as an integral equation for every $t\in[(a-1)\log b,0]$:
		\begin{multline}
			r(t)=r(0)[r_1(t)r_2'(0)-r_1'(0)r_2(t)]+r'(0)[r_1(0)r_2(t)-r_1(t)r_2(0)] \\
			+\int_t^0 e^{(d-4)t'}[r_1(t)r_2(t')-r_1(t')r_2(t)][f_c(t')+g_c(t')r(t')]dt',
			\label{eq:r_int}
		\end{multline}
		where we have used the normalization of the Wronskian 
		$W(r_1,r_2)(t) = e^{-(d-4)t}$ between the two solutions $r_1$ and $r_2$ 
		as in (\ref{Wronskian-r}).
		
		In order to elliminate the divergent behavior of the kernel in \eqref{eq:r_int} as $t\to-\infty$, we introduce the transformation $\tilde{r}(t)=e^{-\kappa_-t}r(t)$, which results in the following integral equation for $\tilde{r}$:
		\begin{multline}
			\tilde{r}(t)=r(0)e^{-\kappa_-t}[r_1(t)r_2'(0)-r_1'(0)r_2(t)]+r'(0)e^{-\kappa_-t}[r_1(0)r_2(t)-r_1(t)r_2(0)] \\	+\int_t^0 K_2(t,t')[e^{-\kappa_-t'}f_c(t')+g_c(t')\tilde{r}(t')]dt',
			\label{eq:r_int_tilde}
		\end{multline}
		where the kernel $K_2(t,t')$ is the same as in (\ref{ker-2}):
		\begin{equation}
			K_2(t,t'):=e^{\kappa_-(t'-t) - (\kappa_+ + \kappa_-)t'}\left[r_1(t)r_2(t')-r_1(t')r_2(t)\right].
		\end{equation}
		The kernel is bounded for every $-\infty < t \leq t' \leq 0$ as in 
		(\ref{bound-ker-2}). It follows from (\ref{bound-on-positive}) that 
		\begin{equation}
			|r(0)|+|r'(0)| \leq  C_0(|\lambda-\lambda_\infty|+|c-c_\infty|)\leq C_0\epsilon b^{\kappa_-(1-a)},
			\label{eq:r0_est}
		\end{equation}
		where the $(\epsilon,b)$-independent constant $C_0$ can change from one line to another line. It follows from the expansion (\ref{sol-Psi-Infty-der}) in Assumption \ref{assumption-2} that $\partial_c\Psi_\infty(t)=\mathcal{O}(e^{\kappa_-t})$ as $t\to-\infty$. Therefore, we get by using bounds \eqref{eq:Psi_c_bound_neg_t}  and (\ref{uniform-bound}):
		\begin{equation}
			\int_t^0 e^{-\kappa_-t'}|f_c(t')|dt'\leq C_0\left(\epsilon b^{\kappa_-(1-a)}\int_{(a-1)\log b}^0 e^{\kappa_-t'}dt'+|\lambda-\lambda_\infty|\right)\leq C_0\epsilon.
			\label{eq:source_est}
		\end{equation}
		On the other hand, for every $\tilde{r}\in L^\infty((a-1)\log b,0)$, we get by using bounds \eqref{eq:Psi_c_bound_neg_t}  and (\ref{uniform-bound}):
		\begin{equation}
			\int_{(a-1)\log b}^0 |g_c(t')\tilde{r}(t')|dt'\leq C_0\epsilon||\tilde{r}||_\infty.
			\label{eq:g_r_tilde_est}
		\end{equation}
		Putting estimates \eqref{r1-r2-main-text}, \eqref{eq:r0_est}, \eqref{eq:source_est}, and \eqref{eq:g_r_tilde_est} together in the integral equation (\ref{eq:r_int_tilde}) yields
		\begin{equation}
		\label{bound-tilde-r}
			\sup_{t\in[(a-1)\log b,0]}|\tilde{r}(t)|\leq C_0\epsilon,
		\end{equation}
		which is the first part of bound \eqref{eq:r_est_intermediate} after going back to the original variable $r(t)$. The second part of bound (\ref{eq:r_est_intermediate}) is obtained by differentiating (\ref{eq:r_int_tilde}) in $t$ and using bound (\ref{bound-tilde-r}).
	\end{proof}
	
	Bound (\ref{eq:r_est_intermediate}) of Lemma \ref{lemma:part_c_Psi_c_est} 
	and the expansion (\ref{sol-Psi-Infty-der}) imply the following important 
	representation of $\partial_c \Psi_c(t)$ at $t = T + (a-1) \log b$.
	
	\begin{corollary}
		Under Assumption \ref{assumption-2}, for every $a \in (0,1)$ and $T > 0$, there exist some $(T,a)$-dependent constants $b_{T,a} > 0$ and $C_{T,a} > 0$ such that for every $b \in (b_{T,a},\infty)$ we have  
		\begin{align} 
			|\partial_c\Psi_c(T+(a-1)\log b)-L_\infty e^{\kappa_- T}b^{-\kappa_-(1-a)}| \leq C_{T,a} \max\{\epsilon b^{-\kappa_-(1-a)},b^{-\kappa_+(1-a)},b^{-(2+\kappa_-)(1-a)}\}.
			\label{eq:part_c_Psi_infty_der_bound_evaluated}
		\end{align}
		\label{corollary-2}
	\end{corollary}
	
	\begin{proof}	
		Bound \eqref{eq:part_c_Psi_infty_der_bound_evaluated} follows from 
		the bound (\ref{eq:r_est_intermediate}) at $t = T + (a-1) \log b$ 
		for fixed $T > 0$, $a \in (0,1)$, and sufficiently large $b > 0$ 
		after $\partial_c \Psi_{\infty}(t)$ for large negative $t$ 
		is expressed from the expansion \eqref{sol-Psi-Infty-der}.
	\end{proof}
	
	\begin{remark}
		Lemma \ref{lemma:part_c_Psi_c_est} can be obtained 
		from the $C^1$ property of $\Psi_c$ in $(\lambda,c)$ after some transformations. It follows from the proof of Proposition \ref{lemma:Psi_c_neg_t} in Appendix \ref{appendix-a}	that 
		$$
		|\Psi_c(t) - \Psi_{\infty}(t)| =  \mathcal{O}(\epsilon) \quad t \in [(a-1)\log b,0],
		$$
		and the asymptotic expansion can be differentiated in $\epsilon$. 
		Parameter $\epsilon$ determines the size of the distance $|\lambda - \lambda_{\infty}|$ and $|c - c_{\infty}|$ so that we can write 
	$c-c_{\infty} = \mathcal{O}(\epsilon b^{-\kappa_-(1-a)})$ and differentiate it in $\epsilon$. 
		By taking derivative in $c$ and using the chain rule, 
		this yields 
		$$
		|\partial_c \Psi_c(t) - \partial_c \Psi_{\infty}(t)| =  \mathcal{O}(\epsilon b^{-\kappa_-(1-a)}) \quad t \in [(a-1)\log b,0],
		$$
		which is equivalent to the bound (\ref{eq:r_est_intermediate}). 
		Since taking derivatives in $c$ and using the chain rule are not obvious 
		from the proof of Proposition \ref{lemma:Psi_c_neg_t}, we gave the precise proof of Lemma \ref{lemma:part_c_Psi_c_est}.
	\end{remark}

	\section{Proofs of the main results}
	\label{sec-proofs}
	
	We recall that $\mathcal{M}_b = \mathcal{M}_{c(b)}$ for $\lambda = \lambda(b)$ since $\Psi_b(t) = \Psi_{c(b)}(t)$ for every $t \in \mathbb{R}$. Hence, both $\partial_b \Psi$ and 
	$\partial_c \Psi_{c(b)}$ are solutions of the same homogeneous 
	equation $\mathcal{M}_b \gamma = 0$ for $\lambda = \lambda(b)$. The following lemma shows that these two solutions are linearly independent 
	for sufficiently large values of $b$.
	
	\begin{lemma}
		Fix $d\geq 13$. Under Assumptions \ref{assumption-1} and \ref{assumption-2}, for every $a \in (0,a_0)$ with $a_0$ given by (\ref{a-0}), there exists $b_0>0$ such that for 
		every $b \in (b_0,\infty)$, there exists no $C \in \mathbb{R}$ 
		such that 
		\begin{equation}
			\label{eq:proportionality_part_c_b_Psi}
			C \partial_c\Psi_{c(b)}(t) = \partial_b\Psi_b(t), \quad \text{for all } t\in\mathbb{R}.
		\end{equation}
		\label{lemma-1}
	\end{lemma}
	
	\begin{proof}
		In order to get a contradiction, suppose that relation (\ref{eq:proportionality_part_c_b_Psi}) holds for some constant $C\in\mathbb{R}$. The results of Corollaries \ref{corollary-1} and \ref{corollary-2} apply for $t = T + (a-1) \log b$ for 
		fixed $a \in (0,a_0)$, $T > 0$, and sufficiently large $b$. Substituting bounds \eqref{eq:part_b_Psi_b_bound_evaluated} and \eqref{eq:part_c_Psi_infty_der_bound_evaluated} 
		into \eqref{eq:proportionality_part_c_b_Psi} yields
		\begin{multline}
			C L_\infty e^{\kappa_-T}b^{-\kappa_-(1-a)} 
			\left[ 1 + \mathcal{O}(\epsilon,b^{-(\kappa_+ - \kappa_-) (1-a)},b^{-2 (1-a)}) \right] \\
			= A_0\kappa_+b^{a\kappa_+-1}e^{\kappa_+ T}  \left[ 1 +  \mathcal{O}(b^{-2(1-a)-a \kappa_+},b^{-a(\kappa_+ -\kappa_-)},b^{a\kappa_+}) \right],
			\label{relation}
		\end{multline}
		where we recall that $2(1-a) + a \kappa_+> 0$ if $a \in (0,a_0)$, where $a_0$ is given by (\ref{a-0}) in Corollary \ref{corollary-1}. 
		Since $A_0 \neq 0$ and $L_{\infty} \neq 0$ by Assumptions \ref{assumption-1} 
		and \ref{assumption-2}, we obtain from (\ref{relation}) that 
		\begin{multline}
			C \left[ 1 + \mathcal{O}(\epsilon,b^{-(\kappa_+ - \kappa_-) (1-a)},b^{-2 (1-a)}) \right]  \\
			= L_{\infty}^{-1} A_0 \kappa_+ b^{a(\kappa_+-\kappa_-)+\kappa_- -1} e^{(\kappa_+ -\kappa_-) T} \left[ 1 +  \mathcal{O}(b^{-2(1-a)-a \kappa_+},b^{-a(\kappa_+ -\kappa_-)},b^{a\kappa_+}) \right].
			\label{relation-2}
		\end{multline}
		Since the remainder terms on both sides of (\ref{relation-2}) are smaller 
		than the leading-order terms and $\kappa_+\neq \kappa_-$, this gives a $T$-dependent coefficient $C$, 
		which is a contradiction with the relation (\ref{eq:proportionality_part_c_b_Psi}) for all $t \in \mathbb{R}$ and hence for all $T > 0$. 
	\end{proof}
	
	From Lemma \ref{lemma-1}, we can now prove Theorem \ref{theorem-main} which 
	states that the Morse index of $\mathcal{L}_b : \mathcal{E}\mapsto \mathcal{E}^*$ is finite and is independent of $b$ for every $b \in (b_0,\infty)$.\\
	
	\begin{proof1}{\em of Theorem \ref{theorem-main}}
		For every $b \in (0,\infty)$, the potential $-3 \mathfrak{u}^2_b(r)$ in $\mathcal{L}_b$ is bounded from below on $[0,\infty)$. The Schr\"{o}dinger operator $-\Delta_r + r^2 : \mathcal{E}\mapsto \mathcal{E}^*$ is strictly positive with a purely discrete spectrum. Since $\mathcal{L}_b = -\Delta_r + r^2 - \lambda(b) - 3 \mathfrak{u}_b^2(r)$ is bounded from below, the number of negative eigenvalues (the Morse index)
		of $\mathcal{L}_b : \mathcal{E} \mapsto \mathcal{E}^*$ is finite 
		by Theorem 10.7 in \cite{Sigal}.
		
		It remains to show that the Morse index is independent of $b$ for every $b \in (b_0,\infty)$. Let us recall the Emden--Fowler transformation (\ref{EF}), which relates 
		solutions of $\mathcal{L}_b v = 0$ with solutions of $\mathcal{M}_b \gamma = 0$ by $v(r) = r^{-1} \gamma(\log r)$.
		The spectrum of $\mathcal{L}_b : \mathcal{E}\mapsto \mathcal{E}^*$ includes the zero eigenvalue if and only if there exists $v \in \mathcal{E}$ satisfying $\mathcal{L}_b v = 0$. This is impossible due to Lemma \ref{lemma-1} according to the following arguments. 
		
		As $t\to-\infty$, there exist two linearly independent solutions to $\mathcal{M}_b \gamma = 0$ and the decaying solution is 
		$$
		\partial_b \Psi_b(t) = e^t + \mathcal{O}(e^{3t}), \quad \mbox{\rm as \;\;} t \to -\infty.
		$$
		The other solution is growing as $e^{(3-d)t}$ which corresponds to $v(r) \sim r^{2-d}$ so that 
		$$
		r^{d-1} |v(r)|^2 \sim r^{3-d}
		$$ 
		is not integrable near $r = 0$ for $d \geq 4$. Hence, the corresponding $v(r)$ is not in $L^2_r$ and if there exists nonzero $v \in \mathcal{E}$ satisfying $\mathcal{L}_b v = 0$, then there exists a constant $C_- \neq 0$ such that 
		$$
		v(r) = C_- r^{-1} \partial_b \Psi_b(\log r).
		$$
		As $t \to +\infty$, there exist two linearly independent solutions to $\mathcal{M}_b \gamma = 0$ and the decaying solution is 
		$$
		\partial_c \Psi_{c(b)}(t) \sim e^{\frac{\lambda(b) - d + 2}{2}} e^{-\frac{1}{2} e^{2t}}, \quad \mbox{\rm as \;\;} t \to \infty.
		$$
		The other solution is growing as $e^{\frac{1}{2} e^{2t}}$, which corresponds to $v(r) \sim e^{\frac{1}{2} r^2}$, clearly not in $L^2_r$. If there exists nonzero $v \in \mathcal{E}$ satisfying $\mathcal{L}_b v = 0$, then there exists a constant $C_+ \neq 0$ such that 
		$$
		v(r) = C_+ r^{-1} \partial_c \Psi_{c(b)}(\log r).
		$$
		Since $C_-,C_+ \neq 0$, if there exists nonzero $v \in \mathcal{E}$,
then $\partial_b \Psi_b$ and $\partial_c \Psi_{c(b)}$ are linearly dependent, which results in a contradiction with  Lemma \ref{lemma-1} for every $b \in (b_0,\infty)$. Hence $0 \notin \sigma(\mathcal{L}_b)$ for every $b \in (b_0,\infty)$ so that the Morse index of $\mathcal{L}_b : \mathcal{E}\mapsto \mathcal{E}^*$ is independent of $b$ for every $b \in (b_0,\infty)$.
	\end{proof1}

\vspace{0.25cm}
	
	By Lemma \ref{lemma:part_c_Psi_c_est}, $\partial_c \Psi_c(t)$ converges to 
	$\partial_c \Psi_{\infty}$ on $[(a-1)\log b,\infty)$ as $(\lambda,c) \to (\lambda_{\infty},c_{\infty})$. Each zero of either $\partial_c \Psi_c$ or $\partial_c \Psi_{\infty}$ is simple since they are solutions of the second-order linear homogeneous equations $\mathcal{M}_c \partial_c \Psi_c = 0$ and $\mathcal{M}_{\infty} \partial_c \Psi_{\infty} = 0$. Consequently, 
	the number of nodal domains of $\partial_c \Psi_c$ in $[(a-1)\log b,\infty)$ 
	coincides with that of $\partial_c \Psi_{\infty}$ in $[(a-1)\log b,\infty)$. 
	
	The following lemma shows that  $\partial_c \Psi_{c(b)}$ does not have additional nodal domains in the interval $(-\infty,(a-1) \log b)$
	for sufficiently large $b$.

	\begin{lemma}
		Fix $d\geq 13$. Under Assumption \ref{assumption-2}, for every $a \in (0,1)$, there exists $b_0>0$ such that for every $b \in (b_0,\infty)$, there exists $C_0 > 0$ such that 
		\begin{equation}
			\label{non-vanish}
			|\partial_c \Psi_{c(b)}(t)| \geq C_0, \quad t \in (-\infty,(a-1)\log b).
		\end{equation} 
		\label{lemma-2}
	\end{lemma}
	
	\begin{proof}
		Recall that $\mathcal{M}_b = \mathcal{M}_{c(b)}$  for $\lambda = \lambda(b)$ and 
		$\mathcal{M}_b \partial_b \Psi_b= 0$ with 
		\begin{equation}
			\label{asymp-1}
			\partial_b \Psi_b(t) = e^t + \mathcal{O}(e^{3t}), \quad \mbox{\rm as} \;\; t \to -\infty.
		\end{equation}
		Similar to the proof of Lemma \ref{lemma:part_b_Psi_b_est}, we denote the second linearly independent solution of $\mathcal{M}_b \gamma = 0$ by $\Phi$ 
		and normalize it such that 
		\begin{equation}
			\label{asymp-2}
			\Phi(t) = (2-d)^{-1} e^{(3-d)t} + \mathcal{O}(e^{(5-d)t}), \quad \text{as } t\to-\infty.
		\end{equation}
		We are interested in the behavior of $\partial_c\Psi_{c(b)}$ for $t\in(-\infty,t_0)$, where $t_0:=(a-1)\log b$. It follows from 
		(\ref{sol-Psi-Infty-der}) and (\ref{eq:r_est_intermediate})  that for sufficiently large $b$, we have 
		\begin{align}
		\label{der-1}
			\partial_c\Psi_{c(b)}(t_0) &= L_\infty e^{\kappa_- t_0} \left[ 1+\mathcal{O}\left(\epsilon, e^{(\kappa_+ - \kappa_-) t_0}, e^{2t_0} \right) \right], \\
			\label{der-2}
			\partial_c\Psi'_{c(b)}(t_0) &= L_\infty \kappa_- e^{\kappa_- t_0} \left[ 1+\mathcal{O}\left(\epsilon, e^{(\kappa_+ - \kappa_-) t_0}, e^{2t_0} \right)\right].
		\end{align}
		Since $\partial_{c} \Psi_{c(b)}$ is a linear combination of $\partial_b \Psi_b$ and $\Phi$ by the linear superposition principle, we can express $\partial_c\Psi_{c(b)}$ as
		\begin{multline}
			\partial_c\Psi_{c(b)}(t) = e^{(d-4)t_0}\big[ (\partial_c \Psi_{c(b)}(t_0)\Phi'(t_0)-\partial_c\Psi'_{c(b)}(t_0)\Phi(t_0) )\partial_b\Psi_{b}(t) \\ + \left(\partial_c\Psi'_{c(b)}(t_0)\partial_b\Psi_{b}(t_0)-\partial_c\Psi_{c(b)}(t_0)\partial_b\Psi'_{b}(t_0)\right)\Phi(t) \big],
		\end{multline}
		where we have used the normalization $W(\partial_b \Psi_b,\Phi) = e^{(4-d)t}$ 
		of the Wronskian between the two solutions $\partial_b \Psi_b$ and $\Phi$. Since $t_0 \to -\infty$ as $b \to \infty$, we can use asymptotics (\ref{asymp-1}) and (\ref{asymp-2}) as well as the boundary condtions (\ref{der-1}) and (\ref{der-2}) to obtain for every $t \in (-\infty,t_0)$:
		\begin{multline}
			\partial_c\Psi_{c(b)}(t) = (d-2)^{-1} L_{\infty} e^{\kappa_- t_0} 
			\left[ (\kappa_- + d - 3) e^{t-t_0} + (1-\kappa_-) e^{(3-d) (t-t_0)} \right] \\
			\times \left[ 1+\mathcal{O}\left(\epsilon, e^{(\kappa_+ - \kappa_-) t_0}, e^{2t_0} \right)\right] \left[ 1+\mathcal{O}(e^{2t})\right],
			\label{eq:part_c_Psi_cb_asymptotics_neg_t}
		\end{multline}
		where $1 - \kappa_- > 0$ and 
		$$
		\kappa_- + d - 3 = \frac{1}{2} \left( d - 2 - \sqrt{d^2 - 16 d + 40}\right)> 0
		$$
		for every $d \geq 13$. Thus, the sign of $\partial_c \Psi_{c(b)}(t)$ 
		for every $t \in (-\infty,t_0)$ coincides with the sign of $L_{\infty}$ 
		so that the bound (\ref{non-vanish}) holds. 
	\end{proof}
	
	From Lemmas \ref{lemma:part_c_Psi_c_est} and \ref{lemma-2}, we can now prove Theorem \ref{theorem-main-infty} which 
	states that the Morse index of $\mathcal{L}_b : \mathcal{E}\mapsto \mathcal{E}^*$ coincides with the Morse index of $\mathcal{L}_{\infty} : \mathcal{E}_{\infty} \mapsto \mathcal{E}_{\infty}^*$  for every $b \in (b_0,\infty)$.\\
	
	\begin{proof1}{\em of Theorem \ref{theorem-main-infty}}
		By Sturm's Oscillation Theorem (see, e.g., Theorem 3.5 in \cite{simon}), 
		the Morse index of $\mathcal{L}_b : \mathcal{E}\mapsto \mathcal{E}^*$ coincides with the number of zeros of the function 
		$v(r)$ on $(0,\infty)$ satisfying $\mathcal{L}_b v = 0$ and $v(r) \to 0$ as $r \to \infty$. Due to the Emden--Fowler transformation (\ref{EF}), 
		the number of zeros of $v(r)$ on $(0,\infty)$ coincides 
		with the number of zeros of $\partial_c \Psi_{c(b)}(t)$ on $\mathbb{R}$ 
		since $\mathcal{M}_b \partial_c \Psi_{c(b)} = 0$ and $\partial_c \Psi_{c(b)}(t) \to 0$ as $t \to \infty$. 
		
		By Lemma \ref{lemma-2}, all zeros of $\partial_c \Psi_{c(b)}$ 
		are located in the interval $[(a-1)\log b, \infty)$ for fixed $a \in (0,1)$ 
		and sufficiently large $b > 0$. By Lemma \ref{lemma:part_c_Psi_c_est} 
		and simplicity of the zeros of $\partial_c\Psi_{c(b)}$ and $\partial_c \Psi_{\infty}$, the number of zeros of $\partial_c\Psi_{c(b)}$ and $\partial_c \Psi_{\infty}$ in $[(a-1)\log b, \infty)$ coincides since 
		$\lambda(b) \to \lambda_{\infty}$ and $c(b) \to c_{\infty}$ as $b \to \infty$. 
		All zeros of $\partial_c \Psi_{\infty}$ are located in the interval 
		$[(a-1)\log b,\infty)$ by Assumption \ref{assumption-2} with the expansion 
		(\ref{sol-Psi-Infty-der}) and give the Morse index of $\mathcal{L}_{\infty} : \mathcal{E}_{\infty} \mapsto \mathcal{E}_{\infty}^*$. Hence, the Morse indices of the 
		two operators are equal for every $b \in (b_0,\infty)$.
	\end{proof1}
		
	\appendix
	\section{Proof of Proposition \ref{lemma:Psi_c_neg_t}}
	\label{appendix-a}
		
	Let $\Sigma(t):=\Psi_c(t)-\Psi_\infty(t)$. It follows from \eqref{eq-psi-singular} that $\Sigma$ satisfies the following equation:
	\begin{equation}
		\label{M-infty}
		\mathcal{M}_\infty\Sigma = \mathcal{F}(\Sigma)(t),
	\end{equation}
	where $\mathcal{M}_\infty$ is defined by \eqref{eq:L_infty_op} and
	\begin{equation*}
		\mathcal{F}(\Sigma)(t) := -(\lambda - \lambda_{\infty}) e^{2t} (\Psi_{\infty}(t) + \Sigma(t)) - 3 \Psi_{\infty}(t) \Sigma(t)^2 - \Sigma(t)^3.
	\end{equation*}
	Since $\Psi_\infty(t)\to \sqrt{d-3}$ as $t\to-\infty$, as it follows from (\ref{sol-Psi-infty-asymptotics}), we can pick two linearly independent solutions $r_1$, $r_2$ to $\mathcal{M}_\infty r =0$ such that
	\begin{equation}
		r_1(t)=\mathcal{O}(e^{\kappa_-t}), \quad r_2(t)=\mathcal{O}(e^{\kappa_+t}), \quad \text{ as } t\to-\infty,
		\label{eq:Sigma_1_2_beh_neg_t}
	\end{equation}
	where $\kappa_- < \kappa_+ < 0$ are given by (\ref{kappa-pm}). Using the Liouville's formula, we normalize the Wronskian according to the relation:
	\begin{equation}
		\label{Wronskian-r}
		W(r_1,r_2)(t)=r_1(t) r_2'(t) - r_1'(t) r_2(t)=e^{-(d-4)t}.
	\end{equation}
	By the variation of parameters method, we rewrite the differential equation (\ref{M-infty}) as an integral equation for every $t\in[(a-1)\log b,0]$:
	\begin{multline}
		\Sigma(t)=\Sigma(0)\left[r_1(t)r_2'(0)-r_1'(0)r_2(t)\right]+\Sigma'(0)\left[r_1(0)r_2(t)-r_1(t)r_2(0)\right]
		\\ + \int_t^0e^{(d-4)t'}\left[r_1(t)r_2(t')-r_1(t')r_2(t)\right]
		\mathcal{F}(\Sigma)(t')dt'.
		\label{eq:volterra_Psi_C_monotone}
	\end{multline}
	In order to use Banach fixed-point iterations, we introduce $\tilde{\Sigma}(t):=e^{-\kappa_-t}\Sigma(t)$, which satisfies 
	$\tilde{\Sigma}(t) = \mathcal{A}(\tilde{\Sigma})(t)$, where 
	\begin{align}
		\begin{split}
			\mathcal{A}(\tilde{\Sigma})(t) &=\Sigma(0)e^{-\kappa_-t}\left[r_1(t)r_2'(0)-r_1'(0)r_2(t)\right]+\Sigma'(0)e^{-\kappa_-t}\left[r_1(0)r_2(t)-r_1(t)r_2(0)\right]\\
			&-(\lambda-\lambda_{\infty})\int_t^0K_2(t,t')e^{2t'}\left[e^{-\kappa_-t'}\Psi_{\infty}(t')+\tilde{\Sigma}(t')\right]dt'\\
			&-\int_t^0K_2(t,t')\left[3e^{\kappa_-t'}\Psi_{\infty}(t')\tilde{\Sigma}^2(t')+e^{2\kappa_-t'}\tilde{\Sigma}^3(t')\right]dt',
		\end{split}
		\label{eq:volterra_Psi_C_monotone_tilde}
	\end{align}
	where the kernel $K_2(t,t')$ is defined as
	\begin{equation}
		\label{ker-2}
		K_2(t,t'):=e^{\kappa_-(t'-t) - (\kappa_+ + \kappa_-)t'}\left[r_1(t)r_2(t')-r_1(t')r_2(t)\right].
	\end{equation}
	It follows from \eqref{eq:Sigma_1_2_beh_neg_t} that 
	$K_2(t,t')\sim 1+e^{(\kappa_+-\kappa_-)(t-t')}$ as $t-t' \to-\infty$, which means that there exists some constant $K_0>0$, such that
	\begin{equation}
		\label{bound-ker-2}
		\sup_{-\infty < t \leq t' \leq 0} |K_2(t,t')|\leq K_0.
	\end{equation}
	It follows from (\ref{eq:Psi_c_bound_pos_t}) that 
	there exists some constant $C_0 > 0$ such that 
	\begin{equation*}
		|\Sigma(0)|+|\Sigma'(0)|\leq C_0\left(|\lambda-\lambda_{\infty}|+|c-c_\infty|\right).
	\end{equation*}
	The integral operator $\mathcal{A}(\tilde{\Sigma})$ in (\ref{eq:volterra_Psi_C_monotone_tilde})
	is estimated for every $\tilde{\Sigma}\in L^\infty((a-1)\log b,0)$ as
	\begin{multline}
		\|\mathcal{A}(\tilde{\Sigma})\|_\infty \leq C_0\Big[ |\lambda-\lambda_\infty|+|c-c_\infty|+|\lambda-\lambda_{\infty}|\|\tilde{\Sigma}\|_\infty\\+b^{-\kappa_- (1-a)}\|\tilde{\Sigma}\|^2_\infty+b^{-2\kappa_-(1-a)}\|\tilde{\Sigma}\|^3_\infty \Big].
	\end{multline}
	Similar estimate applies to $\|\mathcal{A}(\tilde{\Sigma}_1) - \mathcal{A}(\tilde{\Sigma}_2)\|_\infty$. 
	The estimates show that the integral operator $\mathcal{A}(\tilde{\Sigma})$ is closed and is a contraction in the ball $B_\delta\subset L^\infty((a-1)\log b,0)$ of the small radius $\delta := 2C_0\epsilon b^{\kappa_-(1-a)}$, provided that $(\lambda,c)$ satisfy the bound (\ref{lambda-c}) and $\epsilon>0$ is sufficiently small. By the Banach fixed-point theorem, there exists a unique fixed point of $\mathcal{A}(\tilde{\Sigma})$ satisfying
	\begin{equation}
		\|\tilde{\Sigma}\|_\infty \leq 2 C_0 \epsilon b^{\kappa_-(1-a)},
	\end{equation}
	which proves the bound \eqref{eq:Psi_c_bound_neg_t} after going back to the original variable $\Sigma$ and redefining $C_0$.


\begin{thebibliography}{99}
		\bibitem{BIZON2021112358} P. Bizon, F. Ficek, D. E. Pelinovsky, and S. Sobieszek, Ground state in the energy super-critical Gross-Pitaevskii equation with a harmonic potential, Nonlinear Analysis {\bf 210} (2021) 112358.
		
		\bibitem{Budd1989} C.J. Budd, Applications of Shilnikov’s theory to semilinear elliptic equations, SIAM J. Math. Anal. 20 (1989) 1069–1080.
		
		\bibitem{Budd_Norbury1987} C. Budd and J. Norbury, Semilinear elliptic equations and supercritical growth, J. Differential Equations 68 (1987) 169–197.
		
		\bibitem{DF} J. Dolbeault and I. Flores, Geometry of phase space and solutions of semilinear elliptic equations in a ball, Trans. AMS
		359 (2007) 4073–4087.
		
		\bibitem{Ficek} F. Ficek, Schr\"{o}dinger--Newton--Hooke system in higher dimensions: stationary states, Phys. Rev. D {\bf 103} (2021), 104062 (13 pages).
		
		\bibitem{F} R.H. Fowler, Further studies of Emden's and similar differential equations, Quart. J. Math. {\bf 2} (1931), 259--288.
				
		\bibitem{Fuk}  R. Fukuizumi, Stability and instability of standing waves for the nonlinear Schr¨odinger equation with harmonic potential,
		Discrete Cont. Dynam. Syst. 7 (2002) 525–544.
		
		\bibitem{GuoWei} Z. Guo and J. Wei, Global solution branch and Morse index estimates of a semilinear elliptic equation with super-critical
		exponent, Trans. Amer. Math. Soc. 363 (2011) 4777–4799.
		
		\bibitem{HO1} M. Hirose and M. Ohta, Structure of positive radial solutions to scalar equations with harmonic potential, J. Differential
		Equations 178 (2002) 519–540.
		
		\bibitem{HO2} M. Hirose and M. Ohta, Uniqueness of positive solutions to scalar field equations with harmonic potential, Funkcial. Ekvac.
		50 (2007) 67–100.
		
		\bibitem{Sigal} P.D. Hislop and I.M. Sigal, {\em Introduction to Spectral Theory with Applications to Schr\"{o}dinger Operators}, 
		Applied Mathematical Sciences {\bf 113} (Springer, New York, 1996)
		
		\bibitem{JL} D. Joseph and T. Lundgren, Quasilinear Dirichlet problems driven by positive sources, Arch. Ration. Mech. Anal. 49 (1973)
		241–269.
		
		\bibitem{KW} O. Kavian and  F. Weissler, Self-similar solutions of the pseudo-conformally invariant nonlinear Schr¨odinger equation,
		Michigan Math. J. 41 (1994) 151--173.
		
		\bibitem{KikuchiWei} H. Kikuchi and J. Wei, A bifurcation diagram of solutions to an elliptic equation with exponential nonlinearity in higher
		dimensions, Proc. Roy. Soc. Edinburgh 148A (2018) 101–122.
		
		\bibitem{kilip1}  R. Killip and M. Visan, Energy-supercritical NLS: critical $H^s$-bounds imply scattering, 
		Comm. Partial Differential Equations {\bf 35} (2010) 945--987.
		
		\bibitem{kilip2} R. Killip and M. Visan, The radial defocusing energy-supercritical nonlinear wave equation in all space dimensions, 
		Proc. Amer. Math. Soc. {\bf 139} (2011) 1805--1817.
		
		\bibitem{kilip3} R. Killip and M. Visan, The defocusing energy-supercritical nonlinear wave equation in three space dimensions, 
		Trans. Amer. Math. Soc. {\bf 363} (2011) 3893--3934.
		
		\bibitem{MP} F. Merle and L. Peletier, Positive solutions of elliptic equations involving supercritical growth, Proc. R. Soc. Edinburgh
		118A (1991) 40–62.
		
		\bibitem{Pel-book} D. E. Pelinovsky, {\em Localization in Periodic Potentials: From Schr\"{o}dinger Operators to the Gross--Pitaevskii
			Equation}, LMS Lecture Note Series {\bf 390} (Cambridge University Press, Cambridge, 2011).

		
		\bibitem{Selem2011} F. Selem, Radial solutions with prescribed numbers of zeros for the nonlinear Schr¨odinger equation with harmonic
		potential, Nonlinearity 24 (2011) 1795–1819.
		
		\bibitem{SK2012} F. Selem and H. Kikuchi, Existence and non-existence of solution for semilinear elliptic equation with harmonic potential
		and Sobolev critical/supercritical nonlinearities, J. Math. Anal. Appl. 387 (2012) 746–754.
		
		\bibitem{Selem2013} F.H. Selem, H. Kikuchi, and J. Wei, Existence and uniqueness of singular solution to stationary {Schr\"{o}dinger} equation with supercritical nonlinearity, Discr. Contin. Dynam. Systems {\bf 33} (2013), 4613--4626.
		
		\bibitem{simon} B. Simon, Sturm oscillation and comparison theorems, in {\em Sturm-Liouville theory} (Birkh\"{a}user, Basel, 2005), pp. 29--43.
		
	\end{thebibliography}
\end{document}